\documentclass[english]{amsart}
\usepackage[T1]{fontenc}
\usepackage[latin1]{inputenc}
\usepackage{graphicx}
\usepackage{amssymb}
\usepackage{amsmath}
\usepackage{amsthm}
\usepackage{color}
\usepackage[all]{xy}
\usepackage{amsfonts}
\usepackage{mathrsfs}
\usepackage{latexsym}
\usepackage{geometry} 
\usepackage{textcomp}
\usepackage[english]{babel}

\theoremstyle{plain}
\newtheorem*{bigtheo}{Theorem}
\newtheorem{theo}{Theorem}[section]
\newtheorem{prop}[theo]{Proposition}
\newtheorem{hypo}[theo]{Hypothesis}
\newtheorem{lemm}[theo]{Lemma}
\newtheorem{exa}[theo]{Example}
\newtheorem{coro}[theo]{Corollary}
\theoremstyle{definition} 
\newtheorem{defi}[theo]{Definition}
\theoremstyle{remark}
\newtheorem{rema}[theo]{Remark}

\title{Analytic continuation on Shimura varieties with $\mu$-ordinary locus}
\author{St\'ephane Bijakowski}
\address{Institut Mathématique de Jussieu, Université Paris 6, 4 place Jussieu, 75005 Paris France}
\email{stephane.bijakowski@orange.fr}
\keywords{Shimura variety, overconvergent modular forms, $\mu$-ordinary locus, canonical subgroups}
\subjclass[2010]{11G18,11F55 (primary), 14G35 (secondary)}

\begin{document}

\begin{abstract}
We study the geometry of unitary Shimura varieties without assuming the existence of an ordinary locus. We prove, by a simple argument, the existence of canonical subgroups on a strict neighborhood of the $\mu$-ordinary locus (with an explicit bound). We then define the overconvergent modular forms (of classical weight), as well as the relevant Hecke operators. Finally, we show how an analytic continuation argument can be adapted to this case to prove a classicality theorem, namely that an overconvergent modular form which is an eigenform for the Hecke operators is classical under certain assumptions. 
\end{abstract}

\maketitle

\tableofcontents

\section*{Introduction}
$ $\\
\indent A modular form is defined as a global section of a certain sheaf on the modular curve. To study congruences between modular forms, one is led to introducing new objects, namely $p$-adic and overconvergent modular forms. These are sections of a sheaf respectively on the ordinary locus of the modular curve, and on a strict neighborhood of the ordinary locus. A lot of work has been done using these objects, one can for example construct families of overconvergent modular forms. \\
\indent It is possible to generalize the definition of $p$-adic and overconvergent modular forms to other varieties. One can for example consider the Hilbert modular variety, or the Siegel variety. The natural definition for an overconvergent modular form is a section of a certain sheaf on a strict neighborhood of the ordinary locus. In greater generality, one can consider Shimura varieties with a non-empty ordinary locus. \\
$ $\\
\indent But for Shimura varieties without ordinary locus, this definition fails. Recall that for Shimura varieties of PEL type, the criterion for the existence of an ordinary locus at $p$ is that $p$ splits completely in its reflex field. For Shimura varieties of PEL type $(C)$ (associated to symplectic groups), the reflex field is $\mathbb{Q}$, so there is always an ordinary locus. But if one considers Shimura varieties of type $(A)$ (associated to unitary groups), then the reflex field may not be $\mathbb{Q}$ and the ordinary locus may be empty. \\
Let us look at an example. Let $F$ be a CM field with totally real subfield $F_0$, and consider the special fiber of the Shimura variety associated to an unitary group with signature $(a_\sigma,b_\sigma)$ at each real place $\sigma$ of $F_0$. Suppose for simplicity that $p$ is a prime number inert in $F_0$, and which splits in $\pi^+ \pi^-$ in $F$. The choice of $\pi^+$ gives an order for the elements of the couple $(a_\sigma,b_\sigma)$ for each real place $\sigma$ of $F_0$. The existence of the ordinary locus at $p$ is then equivalent to the fact that there exist integers $a$ and $b$ such that $a_\sigma = a$ and $b_\sigma = b$ for each real place $\sigma$. The structure of the $p$-torsion of the abelian variety is then well known on the ordinary locus. In our setting, the $p$-divisible group $A[p^\infty]$ splits in $A[(\pi^+)^\infty] \oplus A[(\pi^-)^\infty]$, and $A[(\pi^-)^\infty]$ is the dual of $A[(\pi^+)^\infty]$. On the ordinary locus, the $p$-divisible group $A[(\pi^+)^\infty]$ is an extension of a multiplicative part of height $da$, and an \'etale part of height $db$ ($d$ is the degree of $F_0$ over $\mathbb{Q}$). In particular, there exists a multiplicative subgroup $H_a \subset A[\pi^+]$ of height $da$. Actually, this property characterizes the ordinary locus : if there exists a multiplicative subgroup $H_a$ of height $ha$, then the abelian variety is ordinary at $p$. On the rigid space associated to the Shimura variety (i.e. the generic fiber of its formal completion along its special fiber), then one can define the ordinary locus. There is still a multiplicative subgroup of $A[\pi^+]$ of height $da$ on this locus, and work of Fargues (\cite{Fa2}) shows that this subgroup extends to a canonical subgroup on a strict neighborhood of the ordinary locus. \\
$ $\\
\indent If the $(a_\sigma)$ are not equal to a certain integer $a$, then the ordinary locus is empty. There is always a special locus in the special fiber of the Shimura variety, called the $\mu$-ordinary locus, but the situation is more involved. Suppose we are in the same setting as before, and let us order the elements $a_1 \leq a_2 \leq \dots \leq a_d$. Then the $\mu$-ordinary locus is characterized by the fact that the $p$-divisible group $A[(\pi^+)^\infty]$ has a filtration $0 \subset X_1 \subset X_2 \subset \dots \subset X_{d+1} = A[(\pi^+)^\infty]$, such that $X_{i+1} / X_i$ is a $p$-divisible group of height $d(a_{i+1} - a_i)$ with its structure explicitly described (by convention we set $a_0 = 0$ and $a_{d+1}=a_d+b_d$). The fact that $A$ is $\mu$-ordinary is then also equivalent to the existence of subgroups $0 \subset H_{a_1} \subset \dots H_{a_d} \subset A[\pi^+]$, with $H_{a_i}$ of height $d a_i$, and the structure of $H_{a_{i+1}} / H_{a_i}$ explicitly described. \\
If one looks at the rigid space of the Shimura variety, then one can define the $\mu$-ordinary locus. There are several canonical subgroups in $A[\pi^+]$ on this locus. We lack a good theory for these canonical subgroups, which should be analogous to the one dealt by Fargues in \cite{Fa2}. However, by simple arguments, one can prove the following fact.

\begin{bigtheo}
On a strict neighborhood of the $\mu$-ordinary locus, there exist canonical subgroups $H_{a_1} \subset \dots \subset H_{a_d}$ in $A[\pi^+]$. These subgroups are characterized by the fact their degree (in the sense of \cite{Fa}) is maximal among the subgroups of the same height of $A[\pi^+]$.
\end{bigtheo}

The proof is actually very simple : let us consider $X_{Iw}$ the variety with Iwahori level at $p$, and denote $f : X_{Iw} \to X$ the projection. The $\mu$-ordinary locus is the image by $f$ of the locus where the subgroups $H_{a_i}$ are of maximal degree. Let $0 < \varepsilon < 1/2$, and consider the locus in $X_{Iw}$ where the degree of $H_{a_i}$ is bigger than the maximal degree minus $\varepsilon$. The image by $f$ of this locus is a strict neighborhood of the $\mu$-ordinary locus. The existence of canonical subgroups follows from the definition. Their uniqueness is a simple computation using the properties of the degree function (see the proposition $\ref{group_A}$). \\
$ $\\
\indent If we consider the space $X_{Iw}$, then one can call the locus where the degree of each $H_{a_i}$ is maximal the $\mu$-ordinary-multiplicative locus. It then makes sense to define an overconvergent modular form as the section of a certain sheaf on a strict neighborhood of the $\mu$-ordinary-multiplicative locus. \\
The Hecke algebra at $p$ acts both on the rigid space and on the space of modular forms. In the case of existence of the ordinary locus, there is one relevant Hecke operator, parametrizing complements of the canonical subgroup. In the general case, there will be as many relevant Hecke operators as the number of canonical subgroups. We will denote by $U_{p,a_i}$ these Hecke operators. One can show that these operators increase the degrees of all the subgroups of $A[\pi^+]$, then act on the space of overconvergent modular forms. \\
We can now state the main result of the paper, namely that an overconvergent  modular form, which is an eigenform for the Hecke operaotrs $U_{p,a_i}$ can be analytically continued to the whole variety under a certain assumption, and thus is classical. Let $\kappa$ be a weight ; explicitly, it is a collection of integers $(\kappa_{i,1} \geq \dots \geq \kappa_{i,a_i} , \lambda_{i,a} \geq \dots \geq \lambda_{i,b_i})_{1 \leq i \leq d}$. Let $S = \{ a_1, \dots, a_d \} \cap [1,a_d+b_d-1]$. The cardinality of $S$ is exactly the number of canonical subgroups. Let us note $\Sigma_i = \{j, a_j = a_i \}$ for all $i$.

\begin{bigtheo} [Theorem \ref{final_theo}] 
Let $f$ be an overconvergent modular form of weight $\kappa$. Suppose that $f$ is an eigenform for the Hecke operators $U_{p,a_i}$, with eigenvalue $\alpha_i$ for $a_i \in S$, and that we have the relations
$$n_{i} + v(\alpha_i) < \inf_{j \in \Sigma_i} (\kappa_{j,a_j} + \lambda_{j,b_j})$$
Then $f$ is classical.
\end{bigtheo}

Here $n_i$ is a constant depending on the variety. It actually comes from the normalization factor of the Hecke operator $U_{p,a_i}$. This theorem is a classicality result, analogous to the one proved by Coleman (\cite{Co}) in the case of the modular curve. Actually, this result has been also proven by Buzzard (\cite{Bu}) and Kassaei (\cite{Ka}) using an analytic continuation method, from which we inspire here. Note that there has been extensive work for the classicality problem in the case of the existence of ordinary locus. We can cite the work of Sasaki (\cite{Sa}), Johansson (\cite{Jo}), Tian et Xiao (\cite{T-X}) and Pilloni and Stroh (\cite{P-S}) in the case of Hilbert varieties, and the work of the author (\cite{Bi1}) for more general PEL Shimura varieties (one can also cite the thesis of the author \cite{Bi2} for Shimura varieties with ramification). \\ 
In this introduction, we have assumed that $p$ is inert in $F_0$ and splits in $F$. The group associated to the Shimura variety is then a linear group at $p$. If $p$ is inert in $F$, then the group is an unitary group at $p$. Everything we have said adapts to that context : the description of the $\mu$-ordinary locus, the existence of the canonical subgroups, and the analytic continuation theorem. Note that the geometry is more involved in that case : for example, to define Hecke operators, one has to deal with subgroups of $A[p^2]$. \\
Of course, the assumption that $p$ is inert in $F_0$ is for simplicity, so what we have said can be formulated for any prime $p$ unramified in $F$. In the redaction of the paper, we have tried to formulate propositions valid both in the linear and unitary case as much as possible, but of course we often had to treat separately the proofs.  \\
$ $\\
\indent Let us now talk briefly about the text. In the first part, we introduce the varieties we are dealing with, define the $\mu$-ordinary locus and study the canonical subgroups. In the section $2$ we define the classical and overconvergent modular, as well as the Hecke operators. In the third part, we prove the analytic continuation result. For simplicity, we suppose that the prime $p$ is inert in the sections $2$ and $3$, and the section $4$ briefly shows how to handle the general case. \\
$ $\\
\indent The author would like to thank Beno\^it Stroh, Valentin Hernandez and Vincent Pilloni for helpful discussions.

\section{Shimura varieties of type (A)}

\subsection{The moduli space}

\subsubsection{Shimura datum}

We will introduce the objects needed to define the Shimura variety of unitary type we will work with. We refer to \cite{Ko} section $5$ for more details. \\
Let $F_0$ be a totally real field of degree $d$, and $F$ a $CM$-extension of $F_0$. Let $(U_{\mathbb{Q}},\langle,\rangle)$ be a non-degenerate hermitian $F$-module, and $G$ its automorphism group. For all $\mathbb{Q}$-algebra $R$, we have
$$G(R) = \left\{ (g,c) \in GL_{F} (U_{\mathbb{Q}} \otimes_{\mathbb{Q}} R) \times R^* , \langle gx,gy \rangle =c \langle x,y\rangle \text{ for all } x,y \in U_{\mathbb{Q}} \otimes_{\mathbb{Q}} R \right\} $$

Let $\tau_1, \dots, \tau_d$ be the embeddings of $F_0$ into $\mathbb{R}$, and let $\sigma_i$ and $\overline{\sigma_i}$ be the two embeddings of $F$ into $\mathbb{C}$ extending $\tau_i$.The choice of $\sigma_i$ gives an isomorphism $F \otimes_{F_0, \tau_i} \mathbb{R} \simeq \mathbb{C}$. Let $U_i = U_\mathbb{Q} \otimes_{F_0,\tau_i} \mathbb{R}$. We note $(a_i,b_i)$ the signature of the anti-hermitian structure on $U_i$.
 Then $G_\mathbb{R}$ is isomorphic to 
$$\text{G} \left( \prod_{i=1}^d \text{U}(a_i,b_i) \right)$$
where $a_i + b_i$ is independent of $i$, and is equal to $\frac{1}{2d} $dim$_\mathbb{Q} U_\mathbb{Q}$. We'll call $a+b$ this quantity. \\
We also give ourselves a morphism of $\mathbb{R}$-algebra $h : \mathbb{C} \to $ End$_F U_\mathbb{R}$ such that $\langle h(z)v,w\rangle ~=~\langle v,h(\overline{z})w\rangle$ and $(v,w) \to \langle v,h(i)w\rangle$ is positive definite. This morphism gives a complex structure on $U_\mathbb{R}$ : let $U^{1,0}_{\mathbb{C}}$ be the subspace of $U_\mathbb{C}$ on which $h(z)$ acts by multiplication by $z$. \\
We then have $U^{1,0}_{\mathbb{C}} \simeq \prod_{i=1}^d (\mathbb{C})^{a_i} \oplus \overline{(\mathbb{C})}^{b_i}$ as $F \otimes_{\mathbb{Q}} \mathbb{R} \simeq \oplus_{i=1}^d \mathbb{C}$-module (the action of $\mathbb{C}$ on $(\mathbb{C})^{a_i} \oplus \overline{(\mathbb{C})}^{b_i}$ is the standard action on the first factor and the conjugated action on the second). \\
Let $U$ be an $O_F$-stable lattice of $U_\mathbb{Q}$, and we assume that the pairing $\langle,\rangle$ induces a pairing $U \times U \to \mathbb{Z}$ which is perfect at $p$.
$ $\\
The ring $O_F$ is a free $\mathbb{Z}$-module. Let $\alpha_1, \dots, \alpha_t$ be a basis of this module, and 
$$\text{det}_{U^{1,0}} = f(X_1, \dots, X_t) = \det (X_1 \alpha_1 + \dots + X_t \alpha_t ;U^{1,0}_{\mathbb{C}} \otimes_{\mathbb{C}} \mathbb{C} [X_1, \dots, X_t])$$
We can show that $f$ is polynomial with algebraic coefficients. The number field $E$ generated by its coefficients is called the reflex field. 

\begin{rema}
We chose for simplicity to work with a central algebra. One can easily adapt the arguments here replacing $F$ by a simple algebra $B$ with center $F$.
\end{rema}

\subsubsection{The Shimura variety}

Let us define now the PEL Shimura variety of type (A) associated to $G$. Let $K$ be an extension of $\mathbb{Q}_p$ containing the images of all the embeddings $F \hookrightarrow \overline{\mathbb{Q}_p}$. Assume that $p$ is unramified in $F$. We fix an embedding $E \hookrightarrow K$, so that the coefficients of the polynomial $\det_{U^{1,0}}$ can be seen as elements of $O_K$. We also fix an integer $N \geq 3$ prime to $p$.

\begin{defi}
Let $X$ be the moduli space over $O_K$, which $S$-points are the isomorphism classes of $(A,\lambda,\iota,\eta)$ where
\begin{itemize}
\item $A \to S$ is an abelian scheme
\item $\lambda : A \to A^t$ is a prime to $p$ polarization.
\item $\iota : O_F \to $ End $A$ is compatible with complex conjugation and the Rosati involution, and the polynomials $\det_{U^{1,0}}$ and $\det_{Lie (A)}$ are equal.
\item $\eta : A[N] \to U/NU$ is an $O_F$-linear symplectic similitude, which lifts locally for the étale topology in an $O_F$-linear symplectic similitude 
$$H_1 (A,\mathbb{A}_f^p) \to U \otimes_{\mathbb{Z}} \mathbb{A}_f^p$$
\end{itemize}
\end{defi}

The moduli space $X$ is representable by a quasi-projective scheme over $O_K$. We will make a slight hypothesis on the variety we consider.

\begin{hypo} \label{notcurve}
We suppose that we are not in the case $d=1$ and $(a,b)=(1,1)$.
\end{hypo}

This condition is technical, and will ensure that one can neglect the cusps in the definition of the modular forms. The case we exclude corresponds essentially to the modular curve, which is well known.

\subsubsection{Iwahori level}

Let $\pi$ be a prime of $F_0$ above $p$, we will call $f$ the residual degree and note $q:=p^f$. Since $p$ is unramified in $F$, we have two possibilities for the behavior of $p$ in $F$ : 
\begin{itemize}
\item $\pi$ splits in $\pi^+ \pi^-$ in $F$. We say that $\pi$ is in case $L$.
\item $\pi$ is inert in $F$. We say that $\pi$ is in case $U$.
\end{itemize}

The terminology $L$ and $U$ becomes from the fact that the group $G$ at $\pi$ is respectively a linear or an unitary group according to the different cases. To define the Iwahori structure, we will break into the two cases.

\begin{defi}
Let $X_{Iw,\pi}$ be the moduli space of isomorphism classes of $(A,\lambda,\iota,\eta,H_\bullet)$, where 
\begin{itemize}
\item $(A,\lambda,\iota,\eta)$ is a point of $X$.
\item $0 \subset H_1 \subset \dots \subset H_{a+b} = A[\pi^+]$, where each $H_i$ is a subgroup of $A[\pi^+]$ which is locally for the étale topology isomorphic to $(O_F / \pi^+)^i$ in the case $L$.
\item $0 \subset H_1 \subset \dots \subset H_{a+b} = A[\pi]$, where each $H_i$ is a subgroup of $A[\pi]$ which is locally for the étale topology isomorphic to $(O_F / \pi)^i$, and such that $H_{a+b-i} = H_i^\bot$ in the case $U$.
\end{itemize}
\end{defi}

The moduli spaces $X_{Iw,\pi}$ are representable by quasi-projective schemes over $O_K$. We also define the full Iwahori space by $X_{Iw} = X_{Iw,\pi_1} \times_X X_{Iw,\pi_2} \times_X \dots \times_X X_{Iw,\pi_g}$, where $\pi_1, \dots, \pi_g$ are the primes of $F_0$ above $p$, and the maps $X_{Iw,\pi_i} \to X$ are the natural morphisms corresponding to forgetting the $(H_i)$. 

\begin{rema}
In the case $L$, the subgroups $A[\pi^+]$ and $A[\pi^-]$ are Cartier duals. This comes from the compatibility between the complex conjugation and the Rosati involution. Therefore, each of these groups is totally isotropic. A flag $(H_\bullet)$ of $A[\pi^+]$ give naturally a flag $(H_\bullet^\bot)$ of $A[\pi^-]$, with $H_i^\bot = (A[\pi^+] / H_i)^D \subset A[\pi^+]^D = A[\pi^-]$. Choosing the prime $\pi^-$ instead of $\pi^+$ would have given the same definition.
\end{rema}

Now we will explicitly describe the determinant condition for the abelian scheme $A$. We are still working with a prime $\pi$ of $F_0$ above $p$, and assume it is of type L. Let $\Sigma_\pi$ be the decomposition group at $\pi$, i.e. the elements $\sigma \in Hom(F_0, \overline{\mathbb{Q}_p})$, such that $\sigma$ sends $\pi$ into the maximal ideal of $\overline{\mathbb{Q}_p}$. For every $\sigma \in \Sigma_\pi$, there are two embeddings $\sigma^+$ and $\sigma^-$ of $F$ into $\overline{\mathbb{Q}_p}$ above $\sigma$ ; the embedding $\sigma^+$ sends $\pi^+$ into the maximal ideal of $\overline{\mathbb{Q}_p}$, and similarly for $\pi^-$. To $\sigma$ we have a couple of integers $(a_\sigma, b_\sigma)$, and the choice of the embedding $\sigma^+$ gives an order for the two elements of the couple. Let $A \to R$ be an abelian scheme over an $O_K$-algebra $R$. If we denote $\omega_\pi := e^* \Omega^1_{A[\pi^\infty]}$, then we have $\omega_\pi = \omega_{\pi}^+ \oplus \omega_{\pi}^-$, where $\omega_\pi^+ := e^* \Omega^1_{A[(\pi^+)^\infty]}$. The determinant condition for $A$ implies then that
$$\omega_\pi^+ = \bigoplus_{\sigma \in \Sigma_\pi} R^{a_\sigma}$$
with $O_F$ acting on $R^{a_\sigma}$ by $\sigma^+$. Similarly, we have
$$\omega_\pi^- = \bigoplus_{\sigma \in \Sigma_\pi} R^{b_\sigma}$$
with $O_F$ acting on $R^{b_\sigma}$ by $\sigma^-$. \\
$ $\\
Now suppose that $\pi$ is of type U. We still denote by $\Sigma_\pi$ the decomposition group at $\pi$ of $F_0$. For each $\sigma \in \Sigma_\pi$, there are two embeddings $\sigma_1$ and $\sigma_2$ of $F$ above $\sigma$ ; the choice of $\sigma_1$ gives an order for the elements of the couple $(a_\sigma, b_\sigma)$. Let $A \to R$ be an abelian scheme over an $O_K$-algebra $R$. If we denote $\omega_\pi := e^* \Omega^1_{A[\pi^\infty]}$, then the determinant condition for $A$ implies
$$\omega_\pi = \bigoplus_{\sigma \in \Sigma_\pi} R^{a_\sigma} \oplus R^{b_\sigma}$$
where $O_F$ acts by $\sigma_1$ on $R^{a_\sigma}$ and by $\sigma_2$ on $R^{b_\sigma}$.

\subsection{$\mu$-ordinary locus}

We will describe in this section the $\mu$-ordinary locus of the Shimura variety. Let us first introduce some notations. \\
Let $L$ be an unramified extension of $\mathbb{Q}_p$ of degree $f_0$, and $k$ a field of characteristic $p$ containing the residue field of $L$. Let $D$ be the Galois group of $L$ over $\mathbb{Q}_p$ ; there is a isomorphism $D \simeq \mathbb{Z} / f_0 \mathbb{Z}$, where $1$ is identified with the Frobenius $\sigma$ of $L$. We will note $W(k)$ the ring of Witt vectors of $k$. Let $\underline{\varepsilon} = (\varepsilon_\tau)_{\tau \in D}$ be a sequence of integers equal to $0$ or $1$. We define a Dieudonné module $M_{\underline{\varepsilon}}$ in the following way : it is a free $W(k)$-module of rank $f_0$, and if $(e_\tau)_{\tau \in D}$ is a basis of this module, then the Frobenius and Verschiebung are defined by 
$$F e_{\sigma^{-1} \tau} = p^{\varepsilon_\tau} e_{\tau}   \qquad V e_{\tau} = p^{1-\varepsilon_\tau} e_{\sigma^{-1} \tau}$$
The module $M_{\underline{\varepsilon}}$ is given an action of the ring of integers of $L$ : $O_L$ acts on $W(k) \cdot e_\tau$ by $\tau$. \\
We'll note $BT_{\underline{\varepsilon}}$ the $p$-divisible group over $k$ corresponding to the Dieudonné module $M_{\underline{\varepsilon}}$, and $H_{\underline{\varepsilon}}$ the $p$ torsion of this $p$-divisible group. 

\subsubsection{Linear case}

Now we come back to our Shimura variety. Consider first the case $L$ ; we are still considering a place $\pi$ of $F_0$ above $\pi$ which splits in $\pi = \pi^+ \pi^-$ in $F$. If $k$ is a field containing the residue field of $O_K$, and if $x=(A,\lambda,\iota,\eta)$ is a $k$-point of $X$, then the fact that the abelian variety $A$ is $\mu$-ordinary at $\pi$ will depend on the $p$-divisible group $A[(\pi^+)^\infty]$. Recall that this $p$-divisible group has an action of $O_{F,\pi^+}$, the completion of $O_F$ at $\pi^+$ ; this is an unramified extension of $\mathbb{Q}_p$ of degree $f$. If $\Sigma_\pi$ denotes as before the decomposition group of $\pi$ in $F_0$, then there is a bijection between $\Sigma_\pi$ and the Galois group of $O_{F,\pi^+}$, and for each $\sigma \in \Sigma_\pi$, we have a couple of integers $(a_\sigma, b_\sigma)$. We order the elements $(a_\sigma)$ by increasing order : we then have $a_1 \leq a_2 \leq \dots \leq a_f$. For each integer $0 \leq i \leq f$ we define the sequence $\underline{\varepsilon}_i = (\varepsilon_{i,j})_{1 \leq j \leq f}$ by $\varepsilon_{i,j} = 1$ if $j \geq i+1$ and $\varepsilon_{i,j} = 0$ otherwise. We also set by convention $a_0 = 0$ and $a_{f+1} = a+b$.

\begin{defi}
Let $k$ be an algebraically closed field of characteristic $p$, and $x=(A,\lambda,\iota,\eta)$ be a $k$-point of $X$. Then $x$ is $\mu$-ordinary at $\pi$ if there is an isomorphism of $p$-divisible groups with $O_{F,\pi^+}$ action
$$A[(\pi^+)^{\infty} ] \simeq \prod_{i=0}^{f}  BT_{\underline{\varepsilon}_i}^{a_{i+1}-a_i}$$
\end{defi}

Remark that the term on the right-hand side is explicitly
$$BT_{(1, \dots, 1)}^{a_1} \times BT_{(0,1,\dots,1)}^{a_2-a_1} \times \dots \times BT_{(0, \dots,0, 1)}^{a_f-a_{f-1}} \times BT_{(0, \dots, 0)}^{b_f}$$

Let $X_0$ denote the special fiber of $X$, and $X_0^{\mu-\pi-ord}$ the $\mu$-ordinary locus at the place $\pi$. We have the following proposition due to Wedhorn (\cite{We} theorem $1.6.2$).

\begin{prop}
The $\mu$-ordinary locus $X_0^{\mu-\pi-ord}$ is open and dense in $X_0$.
\end{prop}

We also have the following characterization of the $\mu$-ordinary locus.

\begin{prop} [\cite{Mo}, Theorem $1.3.7$] \label{EO_l}
Let $k$ be an algebraically closed field of characteristic $p$, and $x=(A,\lambda,\iota,\eta)$ be a $k$-point of $X$. Then $x$ is $\mu$-ordinary at $\pi$ if and only if there is an isomorphism of finite flat group schemes with $O_{F,\pi^+}$ action
$$A[\pi^+] \simeq \prod_{i=0}^{f}  BT_{\underline{\varepsilon}_i}^{a_{i+1}-a_i} [\pi^+]$$
\end{prop}

Since $BT_{\underline{\varepsilon}_i}$ is multiplicative for $i=0$, étale for $i=f$ and bi-infinitesimal otherwise, we have the following criterion for the existence of the ordinary locus at $\pi$.

\begin{prop}
The $\mu$-ordinary locus equals the ordinary locus (at the place $\pi$) if and only if there exists an integer $a$ such that $a_\sigma = a$ for all $\sigma \in \Sigma_\pi$.
\end{prop} 

This last condition is also equivalent to the fact that the local reflex field at $\pi$ is equal to $\mathbb{Q}_p$ (one can see \cite{We} section $1.6$ for more details).

\subsubsection{Unitary case}

Let us now consider the unitary case. Let $k$ be a field containing the residue field of $O_K$, and $x=(A,\lambda,\iota,\eta)$ be a $k$-point of $X$. The fact that the abelian variety $A$ is $\mu$-ordinary at $\pi$ will depend on the $p$-divisible group $A[\pi^\infty]$. Recall that this $p$-divisible group has an action of $O_{F,\pi}$, the completion of $O_F$ at $\pi$ ; this is an unramified extension of $\mathbb{Q}_p$ of degree $2f$. Recall that $\Sigma_\pi$ is the decomposition group at $\pi$ of $F_0$ and it is of cardinal $f$. If $\sigma \in \Sigma_\pi$, there are two embeddings $\sigma_1$ and $\sigma_2$ of $F$ into $\overline{\mathbb{Q}_p}$, and the choice of one of the two gives elements $a_\sigma$ and $b_\sigma$. We suppose that the choice is made such that $a_\sigma \leq b_\sigma$ ; we also order the elements in $\Sigma_\pi$ such that the sequence $(a_\sigma)$ is increasing. We then have
$$a_1 \leq a_2 \leq \dots \leq a_f \leq b_f \leq b_{f-1} \leq \dots \leq b_1$$
This gives an order on the embeddings of $F$ into $\overline{\mathbb{Q}_p}$. For each integer $0 \leq i \leq 2f$ we define the sequence $\underline{\varepsilon}_i = (\varepsilon_{i,j})_{1 \leq j \leq 2f}$ par $\varepsilon_{i,j} = 1$ if $j \geq i+1$ and $\varepsilon_{i,j} = 0$ otherwise. We define a sequence $(\alpha_i)_{0 \leq i \leq 2f+1}$ by $\alpha_0 = 0$, $\alpha_i = a_i$ for $1 \leq i \leq f$, $\alpha_i = b_{2f+1-i}$ for $f+1 \leq i \leq 2f$ and $\alpha_{2f+1} = a+b$.

\begin{defi}
Let $k$ be an algebraically closed field of characteristic $p$, and $x=(A,\lambda,\iota,\eta)$ be a $k$-point of $X$. Then $x$ is $\mu$-ordinary at $\pi$ if there is an isomorphism of $p$-divisible groups with $O_{F,\pi}$ action
$$A[\pi^{\infty} ] \simeq \prod_{i=0}^{2f}  BT_{\underline{\varepsilon}_i}^{\alpha_{i+1}-\alpha_i}$$
\end{defi}

Remark that the term in the right-hand side is explicitly
$$BT_{(1, \dots, 1)}^{a_1} \times BT_{(0,1,\dots,1)}^{a_2-a_1} \times \dots \times BT_{(0, \dots,0,1,1 \dots, 1)}^{a_f-a_{f-1}} 
\times BT_{(0, \dots,0,1 \dots, 1)}^{b_f-a_f} \times BT_{(0, \dots,0,0,1 \dots, 1)}^{a_f-a_{f-1}} \times \dots \times BT_{(0, \dots, 0)}^{a_1}$$

\noindent (We have used the fact that $b_i - b_{i-1} = a_{i-1} - a_i$, since the quantity $a_j+b_j$ is independent of $j$.) \\
Let $X_0$ denote the special fiber of $X$, and $X_0^{\mu-\pi-ord}$ the $\mu$-ordinary locus. We have the following proposition due to Wedhorn (\cite{We} theorem $1.6.2$).

\begin{prop}
The $\mu$-ordinary locus $X_0^{\mu-\pi-ord}$ is open and dense in $X_0$.
\end{prop}

We also have the following characterization of the $\mu$-ordinary locus.

\begin{prop} [\cite{Mo}, Theorem $1.3.7$] \label{EO_u}
Let $k$ be an algebraically closed field of characteristic $p$, and $x=(A,\lambda,\iota,\eta)$ be a $k$-point of $X$. Then $x$ is $\mu$-ordinary at $\pi$ if and only if there is an isomorphism of finite flat group schemes with $O_{F,\pi}$ action
$$A[\pi] \simeq \prod_{i=0}^{2f}  BT_{\underline{\varepsilon}_i}^{\alpha_{i+1}-\alpha_i} [\pi]$$
\end{prop}

Since $BT_{\underline{\varepsilon}_i}$ is multiplicative for $i=0$, étale for $i=2f$ and bi-infinitesimal otherwise, we have the following criterion for the existence of the ordinary locus at $\pi$.

\begin{prop}
The $\mu$-ordinary locus equals the ordinary locus (at the place $\pi$) if and only if $a_\sigma = b_\sigma = (a+b)/{2}$ for all $\sigma \in \Sigma_\pi$.
\end{prop}

Again, this last condition is equivalent to the fact that the local reflex field at $\pi$ is equal to $\mathbb{Q}_p$. \\
$ $\\
We'll need later to work with the rigid space associated to $X$. Let us note $X_{rig}$ this rigid space ; it is the generic fiber of the formal completion of $X$ along its special fiber. We refer to \cite{Be} for more details on rigid spaces. We have a specialization map $sp : X_{rig} \to X_0$, and we'll note $X_{rig}^{\mu-\pi-ord}$ the inverse image of the $\mu$-ordinary locus under the specialization map. We'll also note $X_{Iw,rig}$ the rigid space associated to $X_{Iw}$.

\subsection{Canonical subgroups}

\subsubsection{Degrees and partial degrees}

Before introducing the canonical subgroups on the $\mu$-ordinary locus, we'll define the degree for a finite flat group scheme defined over a finite extension of $\mathbb{Q}_p$, and the partial degrees for these endowed with an action of a ring of integers of an unramified extension of $\mathbb{Q}_p$. 

\begin{defi}
Let $L$ be a finite extension of $\mathbb{Q}_p$ and $G$ be a finite flat group scheme of $p$-power order over $O_L$. Let $\omega_G$ be the conormal module along the unit section. The the degree of $G$ is by definition
$$\deg G := v(\text{Fitt}_0 \text{ } \omega_G)$$
where $Fitt_0$ denotes the Fitting ideal, and the valuation of an ideal is defined by $v(xO_L) = v(x)$, normalized by $v(p)=1$.
\end{defi}

We now state some properties of this function. We refer to the work of Fargues (\cite{Fa} section $3$) for more details.

\begin{prop} \label{degreeprop}
The degree function has the following properties.
\begin{itemize}
\item Let $G$ be as before. Then, if $G^D$ denotes the Cartier dual of $G$, we have 
$$\deg G^D = \text{ht} G - \deg G$$
In particular, $\deg G \in [0,$ ht $G ]$. 
\item The degree function is additive : if we have an exact sequence 
$$0 \to G_1 \to G_2 \to G_3 \to 0$$
with $G_i$ finite flat, then $\deg G_2 = \deg G_1 + \deg G_3$. 
\item Let $G$ and $G'$ be two finite flat group schemes, and suppose that there exists a morphism $f : G \to G'$ which is an isomorphism in generic fiber. Then $\deg G \leq \deg G'$, and we have equality if and only if $f$ is an isomorphism. 
\end{itemize}
\end{prop}

We deduce from the last property the following corollary.

\begin{coro}
Let $G$ be a finite flat group scheme of $p$-power order defined over a finite extension of $\mathbb{Q}_p$. Suppose that $H_1$ and $H_2$ are two finite flat subgroups of $G$. Then we have
$$\deg H_1 + \deg H_2 \leq \deg (H_1 + H_2) + \deg (H_1 \cap H_2)$$
\end{coro}

\begin{proof}
By dividing everything by $H_1 \cap H_2$, we are reduced to the case $H_1 \cap H_2 = \{ 0 \}$. The morphism $H_1 \times H_2 \to H_1 + H_2$ is an isomorphism in generic fiber, thus by the previous proposition, we get
$$\deg (H_1 \times H_2) \leq \deg (H_1 + H_2)$$
But since the degree function is additive, we have $\deg (H_1 \times H_2) = \deg H_1 + \deg H_2$.
\end{proof}

Let $G$ be as in the previous definition and suppose now that $G$ has an action of $O_M$, where $M$ is a finite unramified extension of $\mathbb{Q}_p$. Let $\Sigma$ be the set of embeddings of $M$ into $\overline{\mathbb{Q}_p}$. The the module $\omega_G$ has an action of $O_M$ and has the decomposition  
$$\omega_G = \bigoplus_{\sigma \in \Sigma} \omega_{G,\sigma}$$
where $O_M$ acts on $\omega_{G,\sigma}$ by $\sigma$.

\begin{defi}
The partial degree of $G$ is defined for all $\sigma \in \Sigma$ as 
$$\deg_\sigma G = v(\text{Fitt}_0 \text{ } \omega_{G, \sigma}) $$
\end{defi}

\begin{prop}
The partial degree functions have the following properties.
\begin{itemize}
\item We have $\deg G = \sum_{\sigma \in \Sigma} \deg_\sigma G$.
\item Suppose that $G$ has height $[M : \mathbb{Q}_p] h$. If $G^D$ denotes the Cartier dual of $G$, we have for all $\sigma \in \Sigma$ 
$$\deg_\sigma G^D = h - \deg_\sigma G$$
In particular, $\deg_\sigma G \in [0,h]$. 
\item The partial degree functions are additive : if we have an exact sequence 
$$0 \to G_1 \to G_2 \to G_3 \to 0$$
with $G_i$ finite flat with action of $O_M$, then $\deg_\sigma G_2 = \deg_\sigma G_1 + \deg_\sigma G_3$ for all $\sigma \in \Sigma$. 
\end{itemize}
\end{prop}

We refer to \cite{Bi2} section $1.1.5$ for more details.

\begin{rema}
If we are in the situation of the third point of the proposition $\ref{degreeprop}$, i.e. if we have a morphism of finite flat groups schemes $f : G \to G'$, which is an isomorphism in generic fiber and if $G$ and $G'$ have an action of $O_M$, then it is not true that the partial degrees increase. Indeed, the functions that increase are linear combinations of the partial degrees. For example, if $[M : \mathbb{Q}_p]=2$, there are two partial degrees $\deg_1$ and $\deg_2$, and the functions that increase are $\deg_1 + p \deg_2$ and $p \deg_1 + \deg_2$. See \cite{Bi2} proposition $1.1.33$ for more details.
\end{rema}

\begin{exa}
Let us apply what we have said to our Shimura variety. Let $x=(A,\lambda,\iota,\eta)$ be an $O_L$-point of $X$ (where $L$ is a finite extension of $\mathbb{Q}_p$) and suppose that $\pi$ is a place of $F_0$ above $p$ in case $L$. Then $A[\pi^+]$ has an action of $O_{F_0, \pi}$. Moreover, we have 
$$\deg_\sigma A[\pi^+] = a_\sigma$$
for all $\sigma \in \Sigma_{\pi}$. If $H$ is an $O_F$-stable subgroup of $A[\pi^+]$ of height $fh$, then the orthogonal $H^\bot$ is subgroup of $A[\pi^-]$ of height $f(a+b-h)$. We have $H^\bot \simeq (A[\pi^+] /H)^D$, and thus
$$\deg_{\sigma} H^\bot = (a+b-h) - (a_\sigma - \deg_\sigma H) = b_\sigma - h + \deg_\sigma H$$
for all $\sigma \in \Sigma_\pi$. We see that one has the inequalities
$$\deg_\sigma H \geq h - b_\sigma \quad \text{ and } \quad \deg_\sigma H^\bot \geq b_\sigma - h$$
for all $\sigma \in \Sigma_\pi$.
\end{exa}

\begin{exa}
Suppose now that $\pi$ is in case $U$. Then the group scheme $A[\pi]$ has an action of $O_{F,\pi}$. Recall if $\sigma \in \Sigma_{\pi}$ is an embedding of $F_0$ into $\overline{\mathbb{Q}_p}$ above $\pi$, then there are two embeddings $\sigma_1$ and $\sigma_2$ of $F$ extending $\sigma$. With our previous conventions, we have
$$\deg_{\sigma_1} A[\pi] = a_\sigma \quad \text{ and } \quad \deg_{\sigma_2} A[\pi] = b_\sigma$$
If $H$ is an $O_F$-stable subgroup of $A[\pi]$ of height $2fh$, then the orthogonal $H^\bot$ is subgroup of $A[\pi]$ of height $2f(a+b-h)$. We have $H^\bot \simeq (A[\pi] /H)^{D,c}$, where the subscript $c$ means that the action of $O_F$ on $(A[\pi] /H)^{D,c}$ is the conjugate of the natural one. This comes from the compatibility between the Rosati involution and the complex conjugation. Thus
$$\deg_{\sigma_1} H^\bot = a_\sigma - h + \deg_{\sigma_2} H  \quad \text{ and } \quad \deg_{\sigma_2} H^\bot = b_\sigma - h + \deg_{\sigma_1} H$$
for all $\sigma \in \Sigma_\pi$. We see that one has the inequalities
$$\deg_{\sigma_1} H \geq h - b_\sigma \quad \text{ and } \quad \deg_{\sigma_2} H \geq h - a_\sigma$$
for all $\sigma \in \Sigma_\pi$.

\end{exa}

\subsubsection{Siegel variety}

Let us now recall some facts for the canonical subgroup for the Siegel variety. \\
Let $g \geq 1$ be an integer, and $\mathcal{A}_g$ the Siegel variety. There is a universal abelian scheme $A$ on $\mathcal{A}_g$. There is also a Hasse invariant $Ha$ on $\mathcal{A}_g$. We quote the main result obtained by Fargues \cite{Fa2} on the canonical subgroup.

\begin{prop}
Let $A$ be an abelian scheme of dimension $g$ defined over $O_L$ ($L$ is a finite extension of $\mathbb{Q}_p$). Suppose that the valuation $w$ of the Hasse invariant is strictly less than $1/2$. Then there is a canonical subgroup $H \subset A[p]$, of height $g$, totally isotropic, with
$$\deg H = g - w$$
\end{prop}

Let $\mathcal{A}_{g,rig}$ be the rigid space associated to $\mathcal{A}_g$. Then the ordinary locus of $\mathcal{A}_{g,rig}$ is defined as the locus where the associated abelian scheme is ordinary ; it is also the locus where the Hasse invariant is invertible. The proposition says that on a strict neighborhood of the ordinary locus, there exists a canonical subgroup of high degree in $A[p]$. We propose a simple reformulation of this property. We have the following observation.

\begin{prop}
Let $A$ be an abelian scheme of dimension $g$ defined over $O_L$ ($L$ is a finite extension of $\mathbb{Q}_p$). There exists at most one subgroup $H$ of height $g$ of $A[p]$ with
$$\deg H > g-\frac{1}{2}$$
\end{prop}

\begin{proof}
Suppose not, and let $H_1$ and $H_2$ be two subgroups, with $\deg H_i > g-1/2$. Then we have 
$$2g - 1 < \deg H_1 + \deg H_2  \leq \deg (H_1+H_2) + \deg (H_1 \cap H_2)$$
But $\deg (H_1 + H_2) \leq \deg A[p] =g$, and since $H_1 \cap~H_2$ is of height $h \leq~g-1$, we have $\deg (H_1 \cap~H_2) \leq~g-1$. We get a contradiction.
\end{proof}

This can be used to prove the existence of the canonical subgroup in the following way. Let $\mathcal{A}_g'$ be the Siegel variety parametrizing a $g$-dimensional abelian scheme with polarization and a subgroup $H$ totally isotropic of height $g$. We have a map $f : \mathcal{A}_g' \to \mathcal{A}_g$ corresponding to forgetting $H$. If we denote $\mathcal{A}_{g,rig}'$ the rigid space associated to $\mathcal{A}_g'$, we still have a morphism $f : \mathcal{A}_{g,rig}' \to \mathcal{A}_{g,rig}$. Define $X_r = \{x \in \mathcal{A}_{g,rig}', \deg H(x) \geq g-r \}$ for any rational $r$, it is an admissible open of $\mathcal{A}_{g,rig}'$. Then the ordinary locus of $\mathcal{A}_{g,rig}$ is $f(X_0)$, and it follows from \cite{Bi1} proposition $4.1.7$ that the $(f(X_r))_{r>0}$ form a basis of strict neighborhoods of the ordinary locus (the map $f$ is finite étale, and the $(X_r)_{r>0}$ are strict neighborhoods of $X_0$). The previous proposition shows that on $f(X_r)$ there is exactly one subgroup of height $g$ of degree greater or equal to $g-r$ for $r<1/2$, this is the canonical subgroup. 

\subsubsection{Linear case}

Now let's get back to our Shimura varieties. Suppose we are in case L. Then we have the following proposition.

\begin{prop} \label{group_A}
Let $L$ be a finite extension of $\mathbb{Q}_p$, and $x=(A,\lambda,\iota,\eta)$ an $O_L$-point of $X$. Let $1 \leq i \leq f$ be an integer. Then there exists at most one subgroup $H \subset A[\pi^+]$ stable by $O_F$ of height $f a_i$ such that
$$\deg H > \sum_{j=1}^f \min(a_j,a_i) - \frac{1}{2}$$ 
\end{prop}

\begin{proof} 
Suppose not, and let $H_1,H_2$ be two such subgroups. Let us denote by $fh$ the height of $H_1 \cap H_2$ ; the height of $H_1 + H_2$ is then $f(2a_i-h)$. We have
$$\deg H_1 + \deg H_2  \leq \deg (H_1+H_2) + \deg (H_1 \cap H_2)$$
But 
$$\deg(H_1 + H_2) \leq \sum_{j=1}^f \min(a_j,2a_i-h) \leq  \sum_{j=1}^i a_j + \sum_{j=i+1}^f (2a_i-h)$$
and
$$\deg (H_1 \cap H_2) \leq \sum_{j=1}^f \min(a_j,h) \leq \sum_{j=1}^{i-1} a_j + \sum_{j=i}^f h$$
We finally get
$$\deg H_1 + \deg H_2 \leq 2\sum_{j=1}^{i-1} a_j + 2 \sum_{j=i+1}^f a_i + a_i +h \leq 2 \sum_{j=1}^f \min(a_j,a_i) - 1$$
since $h \leq a_i-1$. We get a contradiction.
\end{proof}

The proposition shows that there exists at most one subgroup of height $fa_i$ and of big degree. The next proposition shows that if two such subgroups exists (with different heights), then we automatically have an inclusion.

\begin{prop} \label{incl_A}
Let $i <j$ be two integers between $1$ and $f$. Let $x=(A,\lambda,\iota,\eta)$ be an $O_L$-point of $X$, and suppose there exists that for $l \in \{i,j\}$ a subgroup $H_l \subset A[\pi^+]$ stable by $O_F$ of height $f a_l$ such that
$$\deg H_l > \sum_{k=1}^f \min(a_k,a_l) - \frac{1}{2}$$
Then we have $H_i \subset H_j$.
\end{prop}

\begin{proof}
Let $fh$ denote the height of $H_i \cap H_j$. We have the following inequalities.
$$\deg (H_i + H_j) \leq \sum_{k=1}^f \min (a_k,a_i + a_j - h) \leq \sum_{k=1}^{j} a_k + \sum_{k=j+1}^f (a_i+a_j-h)$$
$$\deg (H_i \cap H_j) \leq \sum_{k=1}^f \min (a_k,h) \leq \sum_{k=1}^{i-1} a_k + \sum_{k=i}^f h$$
We the get
$$\deg H_i + \deg H_j \leq \deg (H_i + H_j) + \deg (H_i \cap H_j) \leq 2 \sum_{k=1}^{i-1} a_k + \sum_{k=i}^j (a_k +h) + \sum_{k=j+1}^f (a_i+a_j)$$
If we do not have the inclusion $H_i \subset H_j$, then $h \leq a_i-1$. We then get
$$\deg H_i + \deg H_j \leq \sum_{k=1}^f (\min(a_k,a_i) + \min(a_k,a_j)) - (j-i+1)$$
We get a contradiction with the hypothesis stating that $H_i$ and $H_j$ have big degrees.
\end{proof}

As a consequence, we directly get the existence of canonical subgroups for $\mu$-ordinary abelian scheme. Let $s$ be the cardinal of $\{a_1, \dots, a_f \} \cap [1,a+b-1]$, and denote by $A_1 < \dots < A_s$ the different elements of this set.

\begin{coro} \label{coro_A}
Let $L$ be a finite extension of $\mathbb{Q}_p$, and let $x=(A,\lambda,\iota,\eta)$ be an $O_L$-point of $X$. Assume that $x$ is $\mu$-ordinary (i.e. the special fiber of $A$ is $\mu$-ordinary). Then for any integer $1 \leq k \leq s$, there exists a unique subgroup $H_k \subset A[\pi^+]$ of height $fA_k$, with
$$\deg_\sigma H_k = \min(a_\sigma,A_k)$$
for all $\sigma \in \Sigma_\pi$.
\end{coro}

\begin{proof}
We can work over $\overline{\mathbb{Q}}_p$. Since $A$ is $\mu$-ordinary, we have by \cite{Mo} proposition $2.1.9$ a filtration on $A[(\pi^+)^\infty]$ :
$$ X_0 = 0 \subset X_1 \subset \dots \subset X_{f+1} = A[(\pi^+)^\infty]$$
with $X_i$ $p$-divisible groups such that $(X_{i+1} / X_i) \times \overline{\mathbb{F}}_p \simeq BT_{\underline{\varepsilon}_i}^{a_{i+1}-a_i}$ for $0 \leq i \leq f$. Let $Y_i = X_{i+1} / X_i$ for $0 \leq i \leq f$. Then $Y_i$ is a $p$-divisible group over the ring of integers of $\overline{\mathbb{Q}}_p$. The module $\omega_{Y_i}$ decomposes into  $\oplus_{\sigma \in \Sigma_\pi} \omega_{Y_i, \sigma}$, and each $\omega_{Y_i, \sigma}$ is free over the ring of integers of $\overline{\mathbb{Q}}_p$. Recall that we have chosen an ordering for the set $\Sigma_\pi = \{\sigma_1, \dots, \sigma_f \}$. From the description of the special fiber of $Y_i$, one sees that $\omega_{Y_i, \sigma_j}$ is $0$ if $j \leq i$, and is free of rank $a_{i+1}- a_i$ over the ring of integers of $\overline{\mathbb{Q}}_p$, if $j \geq i+1$. Thus $\deg_{\sigma_j} Y_i [\pi^+]$ is $0$ if $j \leq i$ and $a_{i+1} - a_i$ otherwise. \\
Let $H_i = X_i[\pi^+]$ for $0 \leq i \leq f$; it is a finite flat subgroup of $A[\pi^+]$ of height $f a_i$. Moreover, we have
$$\deg_{\sigma_j} H_i = \sum_{k=1}^i \deg_{\sigma_j} H_{k} / H_{k-1} = \sum_{k=1}^{\min(i,j)} (a_k - a_{k-1}) = a_{\min(i,j)} = \min(a_i,a_j)$$
for all $i$ and $j$ between $1$ and $f$. This gives the existence of the desired subgroups. The uniqueness follows from the previous proposition.
\end{proof}

\begin{rema}
We also have the following description of the canonical subgroups in the special fiber. Let $L$ be a finite extension of $\mathbb{Q}_p$, and let $x=(A,\lambda,\iota,\eta)$ be a $\mu$-ordinary point of $X$ defined over $O_L$. Let $A_s$ denote the special fiber of $A$ ; then the Frobenius acts on $A_s[(\pi^+)^\infty]$, and we can form the subgroups 
$$C_i := (\pi^+)^{f-i} A_s[F^f,(\pi^+)^{f-i+1}]$$
for $1 \leq i \leq f$, where $F$ is the Frobenius. Then we have $0 \subset C_1 \subset \dots \subset C_f \subset A_s[\pi^+]$, and the special fiber of the canonical subgroups are equal to one of the $C_i$. More precisely, we have for $1 \leq k \leq s$, $H_k \times k(L) = C_{r(k)}$, with $k(L)$ the residue field of $L$ and $r(k) = \min (l, a_l=A_k)$.
\end{rema}

We can then define the relevant degree functions on $X_{Iw,\pi}$. For each integer $k$, define 
$$d_k = \sum_{\sigma \in \Sigma_\pi} \min(a_\sigma,A_k)$$
Let $X_{Iw,\pi,rig}$ be the associated rigid space. We define the degree function $Deg : X_{Iw,\pi,rig} \to~\prod_{k=1}^s [0,d_k]$ on this space by

$$Deg (A,\lambda,\iota,\eta,H_\bullet) := (\deg H_{A_k})_{1 \leq k \leq s}$$

We also define the $k$-th degree function by $Deg_k (A,\lambda,\iota,\eta,H_\bullet) := \deg H_{A_k}$ for $1 \leq k \leq s$.

\begin{rema}
The integers $s,d_k$ as well as the functions $Deg$ and $Deg_k$ depend on the place $\pi$. If the context is clear, we choose not to write the the dependence on $\pi$ to lighten the notations.
\end{rema}

We then have the following description of the $\mu$-ordinary locus.

\begin{prop} \label{ordgroup_l}
The space $X_{rig}^{\mu-\pi-ord} \subset X_{rig}$ is exactly the image of $Deg^{-1} (\{d_1\} \times \dots \times \{d_s\})$ by the map $X_{Iw,\pi,rig} \to X_{rig}$.
\end{prop}

\begin{proof}
If $x$ is a $\mu$-ordinary point, it then follows from the previous corollary that there exist subgroups $H_k$ of height $fA_k$ with $\deg H_k = d_k$ for all $1 \leq k \leq s$. Conversely, suppose that $(A,\lambda,\iota,\eta)$ is a point of $X_{rig}$, with $A$ defined over the ring of integers of an extension $L$ of $\mathbb{Q}_p$, and that there exist subgroups $H_k$ of height $f A_k$ with $\deg H_k = d_k$ for all $1 \leq k \leq s$. We want to show that $A$ is $\mu$-ordinary ; it suffices to show that $A[\pi^+]$ has a nice description. We have $\deg_\sigma H_k \leq \min(a_\sigma, A_k)$, for all $\sigma \in \Sigma_\pi$, and since $d_k = \sum_{\sigma \in \Sigma_\pi} \min(a_\sigma, A_k)$, this forces the last inequality to be an equality. Define $H_0=0$, $H_{s+1} = A[\pi^+]$, and let $H_j'$ be a complement of $H_{j-1}$ in $H_{j}$ for all $1 \leq j \leq s+1$. This is possible if the field $L$ is large enough. We claim that for $1 \leq j \leq s+1$
$$H_{j} \simeq H_j' \times H_{j-1}$$
Indeed we have a morphism $H_{j-1} \to H_j / H_j'$, which is an isomorphism in generic fiber. The degree of the image of $H_{j-1}$ in $H_j / H_j'$ thus increases ; but since the degree of $H_{j-1}$ is maximal, it must be an equality. We deduce that $\deg H_j = \deg H_j' + \deg H_{j-1}$, and that the morphism $H_j' \times H_{j-1} \to H_j$ is an isomorphism. \\
We finally get
$$A[\pi^+ ] \simeq H_1' \times H_2' \times \dots \times H_{s+1}'$$
But we can explicitly describe the groups $H_j'$. Indeed, for all $1 \leq j \leq s+1$ and $\sigma \in \Sigma_\pi$, we have (setting $A_0=0$ and $A_{s+1}=a+b$)
$$\deg_\sigma H_j' = \deg_\sigma H_j - \deg_\sigma H_{j-1} = \min(a_\sigma, A_j) - \min(a_\sigma, A_{j-1})$$
This quantity is $0$ if $a_\sigma \leq A_{j-1}$ and $A_j - A_{j-1}$ if $a_\sigma \geq A_j$. Since the height of $H_j'$ is $f(A_j - A_{j-1})$, one can see that the special fiber of $H_j'$ is isomorphic to $BT_{\underline{\varepsilon}_{s(j)}} [\pi^+]^{A_{j} - A_{j-1}}$, where $s(j)$ is the number of $\sigma \in D_\pi$ with $a_\sigma \leq A_{j-1}$ (one can see this by looking at the Dieudonné module associated to the special fiber of $H_j'$ for example). We then conclude that $A$ is $\mu$-ordinary by the proposition $\ref{EO_l}$.  
\end{proof}

\subsubsection{Unitary case}

Suppose now we are in case U. Then we have the following proposition.

\begin{prop}
Let $L$ be a finite extension of $\mathbb{Q}_p$, and $x=(A,\lambda,\iota,\eta)$ an $O_L$-point of $X$. Let $1 \leq i \leq f$ be an integer. Then there exists at most one subgroup $H \subset A[\pi]$ stable by $O_F$ of height $2f a_i$ such that
$$\deg H > \sum_{j=1}^{2f} \min(\alpha_j,a_i) - \frac{1}{2}$$ 
Moreover, if such a subgroup exists, it is totally isotropic.
\end{prop}

\begin{proof}
The proof of the first part of the proposition is exactly the same as in the linear case (see proposition $\ref{group_A}$). To prove that the subgroup is totally isotropic, we will use the same argument as in the proof of the proposition $\ref{incl_A}$. We only need to get a bound for the degree of $H^\bot$. But we have $H^\bot \simeq (A[\pi]/H)^D$, so
$$\deg H^\bot = 2fb_i - \deg (A[\pi] /H) = \deg H + \sum_{j=1}^{2f} (b_i - \alpha_j) > \sum_{j=1}^{2f} (b_i - \alpha_j + \min(\alpha_j, a_i)) - \frac{1}{2}$$
But $b_i - \alpha_j + \min(\alpha_j,a_i) = \min(b_i, a_i + b_i - \alpha_j)$, and since $a_i+b_i$ is constant, we have $a_i+~b_i-~\alpha_j=~\alpha_{2f+1-i}$. In conclusion, we get
$$\deg H^\bot > \sum_{j=1}^{2f} \min(b_i, \alpha_j) - \frac{1}{2}$$
We conclude that $H \subset H^\bot$ by applying directly the proof of the proposition $\ref{incl_A}$ (note that $H^\bot$ is of height $2fb_i$). 
\end{proof}


The proposition shows that there exists at most one subgroup of height $2fa_i$ and of big degree. The next proposition shows that if two such subgroups exists (with different heights), then we automatically have an inclusion.

\begin{prop}
Let $i <j$ be two integers between $1$ and $f$. Let $x=(A,\lambda,\iota,\eta)$ be an $O_L$-point of $X$, and suppose there exists that for $l \in \{i,j\}$ a subgroup $H_l \subset A[\pi]$ stable by $O_F$ of height $2f a_l$ such that
$$\deg H_l > \sum_{k=1}^{2f} \min(\alpha_k,a_l) - \frac{1}{2}$$
Then we have $H_i \subset H_j$.
\end{prop}

\begin{proof}
The proof is the same as in the linear case (see the proposition $\ref{incl_A}$).
\end{proof}

As a consequence, we directly get the existence of canonical subgroups for $\mu$-ordinary abelian scheme. Let $s$ be the cardinal of $\{a_1, \dots, a_f \} \cap [1,(a+b)/2]$, and denote by $A_1 < \dots < A_s$ the different elements of this set.

\begin{coro}
Let $L$ be a finite extension of $\mathbb{Q}_p$, and let $x=(A,\lambda,\iota,\eta)$ be an $O_L$-point of $X$. Assume that $x$ is $\mu$-ordinary (i.e. the special fiber of $A$ is $\mu$-ordinary). Then for any integer $1 \leq k \leq s$, there exists a unique totally isotropic subgroup $H_k \subset A[\pi]$ of height $2fA_k$, with
$$\deg_\sigma H_k = \min(\alpha_\sigma,A_k)$$
for all $\sigma \in \Sigma_\pi$.
\end{coro}

\begin{proof}
The proof is similar to the linear case (see $\ref{coro_A}$).
\end{proof}

\begin{rema}
We also have the following description of the canonical subgroups in the special fiber. Let $L$ be a finite extension of $\mathbb{Q}_p$, and let $x=(A,\lambda,\iota,\eta)$ be a $\mu$-ordinary point of $X$ defined over $O_L$. Let $A_s$ denote the special fiber of $A$ ; then the Frobenius acts on $A_s[\pi^\infty]$, and we can form the subgroups 
$$C_i := \pi^{2f-i} A_s[F^{2f},\pi^{2f-i+1}]$$
for $1 \leq i \leq 2f$, where $F$ is the Frobenius. Then we have $0 \subset C_1 \subset \dots \subset C_{2f} \subset A_s[\pi]$, and the special fiber of the canonical subgroups are equal to one of the $C_i$. More precisely, we have for $1 \leq k \leq s$, $H_k \times k(L) = C_{r(k)}$, with $k(L)$ the residue field of $L$ and $r(k) = \min (l, a_l=A_k)$. Note also that $C_i^\bot = C_{2f+1-i}$.
\end{rema}

We can then define the relevant degree functions on $X_{Iw,\pi}$. For each integer $k$, define 
$$d_k = \sum_{j=1}^{2f} \min(\alpha_j,A_k)$$
Let $X_{Iw,\pi,rig}$ be the associated rigid space. We define the degree function $Deg : X_{Iw,\pi,rig} \to~\prod_{k=1}^s [0,d_k]$ on this space by

$$Deg (A,\lambda,\iota,\eta,H_\bullet) := (\deg H_{A_k})_{1 \leq k \leq s}$$

We also define the $k$-th degree function by $Deg_k (A,\lambda,\iota,\eta,H_\bullet) := \deg H_{A_k}$ for $1 \leq k \leq s$. We then have the following description of the $\mu$-ordinary locus.

\begin{prop} \label{ordgroup_u}
The space $X_{rig}^{\mu-ord} \subset X_{rig}$ is exactly the image of $Deg^{-1} (\{d_1\} \times \dots \times \{d_s\})$ by the map $X_{Iw,\pi,rig} \to X_{rig}$.
\end{prop}

\begin{proof}
If $x$ is a $\mu$-ordinary point, it then follows from the previous corollary that there exist subgroups $H_k$ of height $2f A_k$ with $\deg H_k = d_k$ for all $1 \leq k \leq s$. Conversely, suppose that $(A,\lambda,\iota,\eta)$ is a point of $X_{rig}$, with $A$ defined over the ring of integers of an extension $L$ of $\mathbb{Q}_p$, and that there exist subgroups $H_k$ of height $2f A_k$ with $\deg H_k = d_k$ for all $1 \leq k \leq s$. This implies that we have $\deg_\sigma H_k = \min(\alpha_\sigma , A_k)$ for all $\sigma \in D_\pi$ and $1 \leq k \leq s$. We want to show that $A$ is $\mu$-ordinary ; it suffices to show that $A[\pi]$ has a nice description. Define $H_0=0$, $H_{i} = H_{2s+1-i}^\bot$ for $s+1 \leq i \leq 2s+1$, and let $H_j'$ be a complement of $H_{j-1}$ in $H_{j}$ for all $1 \leq j \leq 2s+1$. This is possible if the field $L$ is big enough. We claim that 
$$H_{j} \simeq H_j' \times H_{j-1}$$
Indeed we have a morphism $H_{j-1} \to H_j / H_j'$, which is an isomorphism in generic fiber. The degree of the image of $H_{j-1}$ in $H_j / H_j'$ thus increases ; but since the degree of $H_{j-1}$ is maximal, it must be an equality. We deduce that $\deg H_j = \deg H_j' + \deg H_{j-1}$, and that the morphism $H_j' \times H_{j-1} \to H_j$ is an isomorphism. \\
We finally get
$$A[\pi ] \simeq H_1' \times H_2' \times \dots \times H_{2s+1}'$$
But we can explicitly describe the groups $H_j'$ using the same proof as in the proposition $\ref{ordgroup_l}$. We then conclude that $A$ is $\mu$-ordinary by the proposition $\ref{EO_u}$.  
\end{proof}

\section{Modular forms and Hecke operators}

\subsection{Modular forms}

Let us now define the modular forms for the Shimura variety $X$. Let $\pi$ be a place of $F_0$ above $p$, and suppose it is in case L. Then we define the $O_F \otimes O_K$-module $St_\pi$ by
$$St_\pi := O_K^{a_\sigma} \oplus O_K^{b_\sigma}$$
where $O_F$ acts on $O_K^{a_\sigma}$ by $\sigma^+$ and on $O_K^{b_\sigma}$ by $\sigma^-$. If $\pi$ is in case $U$, we define the $O_F \otimes O_K$-module $St_\pi$ by
$$St_\pi := O_K^{a_\sigma} \oplus O_K^{b_\sigma}$$
where $O_F$ acts on $O_K^{a_\sigma}$ by $\sigma_1$ and on $O_K^{b_\sigma}$ by $\sigma_2$. Finally, we define the $O_F \otimes O_K$-module $St$ by
$$St = \bigoplus_{\pi} St_\pi$$
where $\pi$ runs over the places of $F_0$ above $p$. If $R$ is an $O_K$-algebra, and if $(A,\lambda,\iota,\eta)$ is a $R$-point of $X$, then the $R \otimes O_F$-module $e^* \Omega^1_{A/R}$ is isomorphic to $St \otimes_{O_K} R$. The sheaf $\omega_A := e^* \Omega^{1}_{A/X}$ is then locally isomorphic to $St \otimes_{O_K} \mathcal{O}_X$ (it is a locally free sheaf on $\mathcal{O}_X$). \\
Define 
$$\mathcal{T} = \text{Isom}_{O_F \otimes \mathcal{O}_X} (St \otimes \mathcal{O}_X, \omega_A)$$
It is a torsor on $X$ under the group defined over $O_K$
$$M=\prod_{\pi \in \mathcal{P}} \prod_{\sigma \in \Sigma_\pi} GL_{a_\sigma} \times GL_{b_\sigma}             $$
where $\mathcal{P}$ is the set of primes of $F_0$ above $\pi$. Let $B_M$ be the upper Borel of $M$, $U_M$ its unipotent radical, and $T_M$ its maximal torus. Let $X(T_M)$ be the character group of $T_M$, and $X(T_M)^+$ the cone of dominant weights for $B_M$. If $\kappa \in X(T_M)^+$, we note $\kappa'=- w_0 \kappa \in X(T_M)^+$, where $w_0$ is the element of highest length in the Weyl group of $M$ relatively to $T_M$. \\
Let $\phi : \mathcal{T} \to X$ be the projection morphism.

\begin{defi}
Let $\kappa \in X(T_M)^+$. The sheaf of modular forms of weight $\kappa$ is $\omega^\kappa =~\phi_* O_\mathcal{T}[\kappa']$, where $\phi_* O_\mathcal{T}[\kappa']$ is the subsheaf of $\phi_* O_\mathcal{T}$ where $B_M=T_M U_M$ acts by $\kappa'$ on $T_M$ and trivially on $U_M$.
\end{defi} 

A modular form of weight $\kappa$ on $X$ with coefficients in an $O_L$-algebra $R$ is thus a global section of $\omega^\kappa$, so an element of $H^0(X \times_{O_K} R , \omega^\kappa)$. Using the projection $X_{Iw} \to X$, we define similarly the sheaf $\omega^\kappa$ on $X_{Iw}$, as well as the modular forms on $X_{Iw}$. 

\subsection{Overconvergent modular forms}

For simplicity, we will now assume that there is only one place $\pi$ of $F_0$ above $p$, that is to say that $p$ is inert in $F_0$. The case with several places does not add any difficulty, and will be treated in section \ref{several}. \\ 
We can then define the space of overconvergent modular forms. These will be sections of the sheaf of modular forms defined over a strict neighborhood of the $\mu$-ordinary locus. Recall that we have defined in both cases a degree function.

$$Deg : X_{Iw,rig} \to \prod_{k=1}^s [0,d_k]$$

Since there is only one place above $p$ in $F_0$, we have $X_{Iw}=X_{Iw,\pi}$. \\
We define the $\mu$-ordinary-multiplicative locus as $Deg^{-1} (\{d_1\} \times \dots \times \{d_s\})$. By the proposition $\ref{ordgroup_l}$ or $\ref{ordgroup_u}$, this locus lies in the $\mu$-ordinary locus.

\begin{defi}
The space of overconvergent modular forms of weight $\kappa$ is defined as 
$$M^\dagger := colim_\mathcal{V} H^0(\mathcal{V}, \omega^\kappa)$$
where $\mathcal{V}$ runs over the strict neighborhoods of the $\mu$-ordinary-multiplicative locus in $X_{Iw,rig}$.
\end{defi}

An overconvergent modular form is then defined over a space of the form
$$Deg^{-1} ([d_1-\varepsilon,d_1] \times \dots \times [d_s-\varepsilon,d_s])$$
for some $\varepsilon >0$.

\subsection{Hecke operators}

We now define the Hecke operators. These operators will both act on the rigid space, and on the space of modular forms. We will fix the weight $\kappa$. Explicitly, $\kappa$ is a collection of integer
$$\left ((\kappa_{\sigma,1} \geq \dots \geq \kappa_{\sigma,a_\sigma}),(\lambda_{\sigma,1} \geq \dots \geq \lambda_{\sigma,b_\sigma}) 
\right)_{\sigma \in \Sigma_\pi} $$
We recall that we still assume that $\pi$ is the only place of $F_0$ above $p$. To simplify the notation, we define $\kappa_\sigma := \kappa_{\sigma,a_\sigma}$ and $\lambda_\sigma := \lambda_{\sigma,b_\sigma}$.

\subsubsection{Linear case}

Assume that $\pi$ is in case $L$. Let $1 \leq i \leq a+b-1$ be an integer, and define $C_i$ be the moduli space defined over $K$ parameterizing $(A,\lambda,\iota,\eta,H_\bullet,L)$, with $(A,\lambda,\iota,\eta,H_\bullet)$ a point of $X_{Iw}$ and $L=L_0 \oplus L_0 ^\bot$ is a subgroup of $A[\pi]$, where $L_0$ is an $O_F$-stable subgroup of $A[\pi^+]$ with $A[\pi^+] = H_i \oplus L_0$. We have two morphisms $p_1, p_2 : C_i \to X_{Iw} \times_{O_K} K$. The morphism $p_1$ corresponds to forgetting $L$, and the morphism $p_2$ is defined as $p_2(A,\lambda,\iota,\eta,H_\bullet,L) = (A/L,\lambda',\iota',\eta',H_\bullet')$, with
\begin{itemize}
\item $H_j' = (H_j+L_0)/L_0$ if $j \leq i$.
\item $H_j' = ((\pi^+)^{-1} (H_j \cap L_0)) /L_0$ if $j>i$.
\end{itemize}
We take the polarization $\lambda'$ to be equal to $p \cdot \lambda$, which is a prime to $p$ polarization. Let $C_i^{an}$ be the analytic space associated to $C_i$, and define $C_{i,rig} := p_1^{-1} (X_{Iw,rig})$. The morphisms $p_1,p_2$ give morphisms $C_{i,rig} \to X_{Iw,rig}$.

\begin{defi}
The $i$-th Hecke operator acting on the subsets of $X_{Iw,rig}$ is defined by 
$$U_{\pi,i} (S) = p_2(p_1^{-1} (S))$$
This operator preserves the admissible open subsets, and quasi-compact admissible open subsets.
\end{defi}

Let us denote by $p : A \to A/L$ the universal isogeny over $C_i$. This induces an isomorphism $p^* : \omega_{(A/L)/X} \to \omega_{A/X}$, and thus a morphism 
$p^* (\kappa) :~p_2^* \omega^\kappa \to~p_1^* \omega^\kappa$. For every admissible open $\mathcal{U}$ of $X_{Iw,rig}$, we form the composed morphism
\begin{displaymath}
\widetilde{U}_{\pi,i} :   H^0(U_{\pi,i}(\mathcal{U}),\omega^\kappa) \to H^0 ( p_1^{-1} (\mathcal{U}), p_2^* \omega^\kappa) \overset{p^*(\kappa)}{\to} H^0(p_1^{-1}(\mathcal{U}) , p_1^* \omega^\kappa) \overset{Tr_{p_1}}{\to}  H^0(\mathcal{U},\omega^\kappa)
\end{displaymath}

\begin{defi}
The Hecke operator acting on modular forms is defined by $U_{\pi,i}~=~\frac{1}{p^{N_i}} \widetilde{U}_{\pi,i}$ with 
$$N_i= \sum_{\sigma \in \Sigma_\pi} \Big( \min(i,a_\sigma)\min(a+b-i,b_\sigma) + \max(a_\sigma-i,0) \kappa_\sigma + \max(i-a_\sigma,0) \lambda_\sigma \Big) $$
\end{defi}

We will also write 
$$n_i = \sum_{\sigma \in \Sigma_\pi} \min(i,a_\sigma)\min(a+b-i,b_\sigma)$$
the constant term of $N_i$, which is independent of the weight. \\
Let us explain briefly the meaning of the normalization factor $N_i$. Then term $\min(i,a_\sigma)\min(a+~b-~i,b_\sigma)$ comes from the inseparability degree of the projection $p_1$. The term $\max(a_\sigma-i,0) \kappa_\sigma + \max(i-~a_\sigma,0) \lambda_\sigma$ comes from the morphism $p^* (\kappa)$. Indeed, we have the following proposition. 

\begin{prop} \label{minorate}
Let $M$ be a finite extension of $\mathbb{Q}_p$, let $(A,\lambda,\iota,\eta,H_\bullet)$ be an $O_M$-point of $X_{Iw}$ and $L=L_0 \oplus L_0 ^\bot$  a subgroup of $A[\pi]$, where $L_0$ is an $O_F$-stable subgroup of $A[\pi^+]$ with $A[\pi^+] = H_i \oplus L_0$ in generic fiber. Then we have for all $\sigma \in \Sigma_\pi$
\begin{displaymath}
\begin{array}{ccc}
\deg_\sigma L_0 \geq a_\sigma - i  & \text{ and } &  \deg_\sigma L_0^\bot \geq i - a_\sigma 
\end{array}     
\end{displaymath}
\end{prop}

\begin{proof}
The group $A[\pi^+]/L_0$ is of height $fi$, and has partial degree $(a_\sigma - \deg_\sigma L_0)_\sigma$. Hence $a_\sigma-~\deg_\sigma L_0 \leq~i$, and $\deg_\sigma L_0 \geq a_\sigma - i$. We get the other equality by duality (note that $b_\sigma - (a+b-i) = i - a_\sigma$).
\end{proof}

We have the following proposition concerning the behavior of the Hecke operator regarding the degree function.

\begin{prop}
Let $x=(A,\lambda,\iota,\eta,H_\bullet)$ be a point of $X_{Iw,rig}$, and $y \in U_{\pi,i}(x)$ corresponding to a subgroup $L \in A[\pi]$ as before. Write $y=(A/L,\lambda,\iota,\eta,H_\bullet')$ ; then we have
$$\deg H_j' \geq \deg H_j$$
for all $1 \leq j \leq a+b-1$. Moreover, we have
$$\deg H_i' = \deg A[\pi^+] - \deg L_0   $$
If $\deg H_i' = \deg H_i$, then $\deg H_i \in \mathbb{Z}$.
\end{prop}

\begin{proof}
For all $1 \leq j \leq i$, the morphism $H_j \to H_j'$ is an isomorphism in the generic fiber. Thus the inequality $\deg H_j' \geq \deg H_j$. If $i < j \leq a+b-1$, we have
$$\deg H_j' = \deg ((\pi^+)^{-1} (H_j \cap L_0)) - \deg L_0$$
Since $\deg ((\pi^+)^{-1} H) = \deg A[\pi^+] + \deg H$ for every subgroup $H$ of $A[\pi^+]$, we get
$$\deg H_j' = \deg A[\pi^+] + \deg (H_j \cap L_0) - \deg L_0 = \deg (H_j+L_0) + \deg (H_j \cap L_0) - \deg L_0 \geq \deg H_j$$
from the properties of the degree function. \\
We have $H_i' = (A[\pi^+] / L_0)$, hence the formula for the degree of $H_i'$. If $\deg H_i' = \deg H_i$, then we have $\deg H_i + \deg L_0 = \deg A[\pi^+]$, and $A[\pi^+] = H_i \times L_0$. Since $A[\pi^+]$ is a $BT_1$, so is $H_i$, and its degree is an integer.
\end{proof}

\subsubsection{Unitary case}

Assume now that $\pi$ is in case $U$. Let $1 \leq i \leq (a+b)/2$ be an integer, and define $C_i$ be the moduli space defined over $K$ parameterizing $(A,\lambda,\iota,\eta,H_\bullet,L)$, with $(A,\lambda,\iota,\eta,H_\bullet)$ a point of $X_{Iw}$ and 
\begin{itemize}
\item $L$ is an $O_F$-stable, totally isotropic subgroup of $A[\pi^2]$ such that $A[\pi] = H_i \oplus L[\pi] = H_i^\bot \oplus \pi L$ if $i < (a+b)/2$.
\item $L$ is an $O_F$-stable, totally isotropic subgroup of $A[\pi]$ such that $A[\pi] = H_i \oplus L$ if $i = (a+b)/2$.
\end{itemize}
We have two morphisms $p_1, p_2 : C_i \to X_{Iw} \times_{O_K} K$. The morphism $p_1$ corresponds to forgetting $L$, and the morphism $p_2$ is defined as $p_2(A,\lambda,\iota,\eta,H_\bullet,L) = (A/L,\lambda',\iota',\eta',H_\bullet')$, with
\begin{itemize}
\item $H_j' = (H_j+L)/L$ if $j \leq i$.
\item $H_j' = (\pi^{-1} (H_j \cap L) + L) /L$ if $i<j \leq (a+b)/2$.
\end{itemize}
We take the polarization $\lambda'$ to be equal to $p \cdot \lambda$, which is a prime to $p$ polarization. Let $C_i^{an}$ be the analytic space associated to $C_i$, and define $C_{i,rig} := p_1^{-1} (X_{Iw,rig})$. The morphisms $p_1,p_2$ give morphisms $C_{i,rig} \to X_{Iw,rig}$.

\begin{defi}
The $i$-th Hecke operator acting on the subsets of $X_{Iw,rig}$ is defined by 
$$U_{\pi,i} (S) = p_2(p_1^{-1} (S))$$
This operator preserves the admissible open subsets, and quasi-compact admissible open subsets.
\end{defi}

\begin{rema}
The condition $A[\pi] = H_i^\bot \oplus \pi L$ is actually redundant with the condition $A[\pi]=~H_i \oplus~L[\pi]$. Indeed, since $L$ is totally isotropic, we have $\pi L \subset L[\pi]^\bot$ (we denote by $L[\pi]^\bot$ the orthogonal in $A[\pi]$ of $L[\pi]$). Comparing the heights, we see that we have the equality $\pi L = L[\pi]^\bot$.
\end{rema}

Let us denote by $p : A \to A/L$ the universal isogeny over $C_i$. This induces an isomorphism $p^* : \omega_{(A/L)/X} \to \omega_{A/X}$, and thus a morphism 
$p^* (\kappa) :~p_2^* \omega^\kappa \to~p_1^* \omega^\kappa$. For every admissible open $\mathcal{U}$ of $X_{Iw,rig}$, we form the composed morphism
\begin{displaymath}
\widetilde{U}_{\pi,i} :   H^0(U_{\pi,i}(\mathcal{U}),\omega^\kappa) \to H^0 ( p_1^{-1} (\mathcal{U}), p_2^* \omega^\kappa) \overset{p^*(\kappa)}{\to} H^0(p_1^{-1}(\mathcal{U}) , p_1^* \omega^\kappa) \overset{Tr_{p_1}}{\to}  H^0(\mathcal{U},\omega^\kappa)
\end{displaymath}

\begin{defi}
The Hecke operator acting on modular forms is defined by $U_{\pi,i}~=~\frac{1}{p^{N_i}} \widetilde{U}_{\pi,i}$ with 
$$N_i= \sum_{\sigma \in \Sigma_\pi} \Big((a+b) \min(i,a_\sigma) +\max(a_\sigma-i,0)\kappa_\sigma + \max(b_\sigma-a_\sigma,b_\sigma-i) \lambda_\sigma \Big) $$
if $i<(a+b)/2$, and
$$N_{(a+b)/2}= \sum_{\sigma \in \Sigma_\pi} \Big(\frac{(a+b)}{2} a_\sigma + \frac{(b_\sigma-a_\sigma)}{2} \lambda_\sigma \Big) $$
\end{defi}

We will also write 
$$n_i = \sum_{\sigma \in \Sigma_\pi} (a+b) \min(i,a_\sigma)$$
if $i < (a+b)/2$, and
$$n_{(a+b)/2}= \sum_{\sigma \in \Sigma_\pi} \frac{(a+b)}{2} a_\sigma $$
the constant term of $N_i$, which is independent of the weight. \\
Again, the reason for the normalization factor $n_i$ comes into two parts. Then term $(a+b) \min(i,a_\sigma)$ comes from the inseparability degree of the projection $p_1$. The term $\max(a_\sigma-i,0)\kappa_\sigma + \max(b_\sigma-~a_\sigma,b_\sigma-~i) \lambda_\sigma$ comes from the morphism $p^* (\kappa)$. Indeed, we have the following proposition. 

\begin{prop} \label{minorate_u}
Let $M$ be a finite extension of $\mathbb{Q}_p$, let $(A,\lambda,\iota,\eta,H_\bullet)$ be an $O_M$-point of $X_{Iw}$ and $L$ a subgroup of $A[\pi^2]$ as before. If $i< (a+b)/2$, we have 
\begin{displaymath}
\begin{array}{ccc}
\deg_{\sigma_1} L \geq a_\sigma - i  & \text{ and } &  \deg_{\sigma_2} L \geq \max(b_\sigma-i,b_\sigma - a_\sigma) 
\end{array}     
\end{displaymath}
If $i=(a+b)/2$, we have
$$\deg_{\sigma_2} L \geq \frac{b_\sigma - a_\sigma}{2}$$ 
\end{prop}

\begin{proof}
Suppose first that $i<(a+b)/2$. The subgroup $L$ being totally isotropic, we get
$$\deg_{\sigma_2} L = \deg_{\sigma_1} L + b_\sigma - a_\sigma$$
The group $A[\pi]/L[\pi]$ is of height $2fi$, hence we have $\deg_{\sigma_1} A[\pi] / L[\pi] \leq i$. We get 
$$\deg_{\sigma_1} L \geq \deg_{\sigma_1} L[\pi] \geq a_\sigma - i$$
We deduce that $\deg_{\sigma_2} L = \deg_{\sigma_1} L + b_\sigma - a_\sigma \geq \max(b_\sigma-i,b_\sigma - a_\sigma)$. \\
If $i=(a+b)/2$, then $L$ is a maximal totally isotropic subgroup of $A[\pi]$. Thus we have
$$\deg_{\sigma_2} L = \frac{b_{\sigma} - a_{\sigma}}{2} + \deg_{\sigma_1} L \geq \frac{b_{\sigma} - a_{\sigma}}{2}$$
\end{proof}

We have the following proposition concerning the behavior of the Hecke operator regarding the degree function.

\begin{prop}
Let $x=(A,\lambda,\iota,\eta,H_\bullet)$ be a point of $X_{Iw,rig}$, and $y \in U_{\pi,i}(x)$ corresponding to a subgroup $L \in A[\pi^2]$ as before. Write $y=(A/L,\lambda,\iota,\eta,H_\bullet')$ ; then we have
$$\deg H_j' \geq \deg H_j$$
for all $1 \leq j \leq (a+b)/2$. Moreover, if $i< (a+b)/2$, we have
$$\deg H_i' = 2fi - \deg (L/L[\pi])$$
If $\deg H_i' = \deg H_i$, then $\deg H_i \in \mathbb{Z}$. \\
If $i=(a+b)/2$, then 
$$\deg H_{(a+b)/2}' = f(a+b) - \deg (L)$$
If $\deg H_{(a+b)/2}' = \deg H_{(a+b)/2}$, then $d_{(a+b)/2} - \deg H_{(a+b)/2} \in 2 \mathbb{Z}$.
\end{prop}

\begin{proof}
Suppose first that $i < (a+b)/2$. For all $1 \leq j \leq i$, the morphisms $H_j \to H_j'$ is an isomorphism in the generic fiber. Thus the inequality $\deg H_j' \geq \deg H_j$. Suppose $i <j \leq (a+b)/2$. Then, observing that $\pi^{-1} (H_j \cap L) \cap L = L[\pi]$, we get
\begin{displaymath}
\begin{split}
\deg H_j' & = \deg (\pi^{-1} (H_j \cap L) + L ) - \deg L \geq \deg (\pi^{-1} (H_j \cap L)) - \deg L[\pi]  \\
& \geq \deg A[\pi] + \deg (H_j \cap L) - \deg L[\pi] = \deg (H_j + L[\pi]) + \deg (H_j \cap L[\pi]) - \deg L[\pi] \\
& \geq \deg H_j
\end{split}
\end{displaymath}
from the properties of the degree function. \\
Let us calculate the degree of $H_i'$. We have
\begin{displaymath}
\begin{split}
\deg H_i' & = \deg (A[\pi] + L ) /L = \deg (\pi^{-1} L[\pi]^\bot)/L = \deg A[\pi] + \deg L[\pi]^\bot - \deg L \\
& = 2fi + \deg L[\pi] - \deg L = 2fi - \deg (L/L[\pi])
\end{split}
\end{displaymath}
If $\deg H_i' = \deg H_i$, then we have $\deg H_i + \deg L[\pi] = \deg A[\pi]$, and $A[\pi] = H_i \times L[\pi]$. Since $A[\pi]$ is a $BT_1$, so is $H_i$, and its degree is an integer. \\
Suppose now that $i=(a+b)/2$. The same argument as before shows that we still have $\deg H_j' \geq~\deg H_j$, for all $1 \leq j \leq (a+b)/2$. Since $H_{(a+b)/2}' = A[\pi] /L$, we have
$$\deg H_{(a+b)/2}' = \deg A[\pi] - \deg L = f(a+b) - \deg (L)$$
If we have the equality $\deg H_{(a+b)/2}' = \deg H_{(a+b)/2}$, then $H_{(a+b)/2}$ is a $BT_1$, and all its partial degrees are integers. Since $H_{(a+b)/2}$ is totally isotropic, we have the relations 
$$\deg_{\sigma_2} H_{(a+b)/2} = \deg_{\sigma_1} H_{(a+b)/2} + (b_\sigma - a_\sigma)/2$$
for all $\sigma \in \Sigma_{\pi}$. If we note $h_\sigma = \deg_{\sigma_1} H_{(a+b)/2}$, which is an integer under the previous assumption, we have
$$d_{(a+b)/2} - \deg H_{(a+b)/2} = \sum_{\sigma \in \Sigma_\pi} \left( a_\sigma + \frac{(a+b)}{2} - 2 h_{\sigma} - \frac{(b_\sigma - a_\sigma)}{2} \right) = 2 \sum_{\sigma \in \Sigma_{\pi}} (a_\sigma - h_\sigma) \in 2 \mathbb{Z}$$
\end{proof}

We will also need the following useful lemma.

\begin{lemm} $\label{isotrop}$
Let $x=(A,\lambda,\iota,\eta,H_\bullet)$ be a point of $X_{Iw,rig}$, and let $L \in A[\pi^2]$ be a totally isotropic subgroup as before, corresponding to a point of $U_{\pi,i} (x)$, with $i < (a+b)/2$.
We have the inequality 
$$\deg (L/L[\pi]) \leq 2fi - \deg A[\pi] + \deg L[\pi]$$
\end{lemm}

\begin{proof}
The multiplication by $\pi$ gives a morphism $L/L[\pi] \to \pi L$. But $L$ being totally isotropic, we have the relation $\pi L = L[\pi]^\bot$. We thus get a morphism $L/L[\pi] \to L[\pi]^\bot$, which is an isomorphism in generic fiber. Thus the relation
$$\deg (L/L[\pi]) \leq \deg L[\pi]^\bot = 2fi - \deg A[\pi] + \deg L[\pi]$$
\end{proof}

\section{A classicality result}

We will now prove a control theorem, that is to say that an overconvergent modular form is indeed classical under a certain assumption.

\subsection{Decomposition of the Hecke operators}

Fix a rational $\varepsilon >0$. We will fix rationals $\varepsilon_k \in \{\varepsilon, d_k \}$ for $1\leq k \leq s$. Fix also an integer $i$ between $1$ and $s$ such that $\varepsilon_i = \varepsilon$. Since we have assumed that there is only one place $\pi$ of $F_0$ above $p$, we will simply note $\Sigma = \Sigma_\pi$. We define a partition of this set. First, we note for $1 \leq k \leq s$
$$\Sigma_k := \{ \sigma \in \Sigma, a_\sigma = A_k \} $$
We will also note $\Sigma_0 := \{ \sigma \in \Sigma, a_\sigma = 0 \}$, and $\Sigma_{s+1} := \{ \sigma \in \Sigma, a_\sigma = a+b \}$ (of course $\Sigma_{s+1}$ is always empty in case U). The sets $(\Sigma_k)_{0 \leq k \leq s+1}$ form a partition of $\Sigma$. \\
From the collection of $(\varepsilon_k)_{1 \leq k \leq s}$, we define another partition of $\Sigma$. The set $S_1$ is defined to be 
\begin{displaymath}
S_1 = \Sigma_0 \cup \bigcup_{\substack{k \neq i \\ \varepsilon_k = \varepsilon} } \Sigma_k \cup \Sigma_{s+1}
\end{displaymath}
The complement is the set
$$S_2 = \Sigma_i \cup \bigcup_{k,\varepsilon_k = d_k } \Sigma_k$$
Define
$$\mathcal{U}_0 := Deg^{-1} ([d_1-\varepsilon_1,d_1] \times \dots \times [d_s-\varepsilon_s,d_s]) = \bigcap_{k=1}^s Deg_k^{-1}  [d_k-\varepsilon_k,d_k]$$
$$\mathcal{U}_1 := \bigcap_{k \neq i} Deg_k^{-1}  [d_k-\varepsilon_k,d_k]$$ 

We will define a decomposition of the Hecke operator $U_{\pi,A_i}$ on subsets of $\mathcal{U}_1$. Fix a rational $\alpha >0$, and define the integer $t$ to be $1$, except in the unitary case and if $i=(a+b)/2$, where we set $t=2$. For simplicity, we will simply write $U_i$ for $U_{\pi,A_i}$. Define
$$\mathcal{U} := Deg_i^{-1} ([0,d_i-t(1-\alpha)]) \cap \mathcal{U}_1$$

\begin{theo} \label{decompo}
Let $N \geq 1$ and $\beta < \varepsilon$ be a positive rational. There exists a finite ordered set $M_N$, and a decreasing sequence of admissible open subsets $(\mathcal{U}_k (N))_{k \in M_N}$ of $\mathcal{U}$ such that for all $k\geq 0$, we have a decomposition of the operator $U_i^N$ on $\mathcal{U}_{k}(N) \backslash \mathcal{U}_{k+1}(N)$ of the form 
$$ U_{i}^N = \left ( \coprod_{j=0}^{N-1} U_i^{N-1-j} \circ T_j  \right ) \coprod T_N$$
with $T_0 = U_{i,k,N}^{good}$, and for $0 < k < N$
$$T_j = \coprod_{k_1 \in M_{N-1}, \dots, k_j \in M_{N-j}}  U_{i,k_j,N}^{good} U_{i,k_{j-1},k_j,N}^{bad} \dots U_{i,k,k_1,N}^{bad}$$
and
$$T_N = \coprod_{k_1 \in M_{N-1}, \dots, k_{N-1} \in M_1} U_{i,k_{N-1},N}^{bad} U_{i,k_{N-2},k_{N-1},N}^{bad} \dots U_{i,k,k_1,N}^{bad}$$  
such that
\begin{itemize}
\item the images of $U_{i,j,N}^{good}$ ($j \in M_k$) are in $Deg_i^{-1} (]d_i -t(1 - \beta),d_i])$.
\item the images of $U_{i,l,l',N}^{bad}$ ($l \in M_k$, $l' \in M_{k-1})$ and $U_{i,l,N}^{bad}$ ($l \in M_1$) are in $Deg_i^{-1} ([0,d_i -t(1 - \beta)])$
\end{itemize}
\end{theo}

The idea is that the points in $]d_i -t(1 - \beta),d_i]$ are `good' (because the overconvergent modular form will be defined at these points), and the points in $Deg_i^{-1} ([0,d_i -t(1 - \beta)])$ are `bad'. The decomposition in the theorem is then made to ensure that all the operators have either their image in good points, or in bad points.

Let us describe this decomposition by looking at a point $x \in \mathcal{U}$. The set $U_i(\{x\})$ is finite, and has say $N_1$ points in $Deg_i^{-1} ([0,d_i -t(1 - \beta)])$ and $N_2$ in its complement $Deg_i^{-1} (]d_i -t(1 - \beta),d_i])$. Thus one can decompose the operator $U_i$ over $x$ as 
$$U_i = U_{i,x}^{good} \coprod U_{i,x}^{bad}$$
with $U_{i,x}^{bad}$ corresponding to the $N_1$ points in $Deg_i^{-1} ([0,d_i -t(1 - \beta)])$, and $U_{i,x}^{good}$ to the $N_2$ other points. This is a decomposition of operators, meaning that the set $U_i(\{x\})$ is the disjoint union of the sets $U_{i,x}^{good} (\{x\})$ and $U_{i,x}^{bad} (\{x\})$ ; moreover the operators $U_{i,x}^{good}$ and $U_{i,x}^{bad}$ induce morphisms $H^0( U_i(\{x\}), \omega^\kappa) \to H^0(\{x\},\omega^\kappa)$ for each weight $\kappa$ such that
$$U_i f = U_{i,x}^{good} f + U_{i,x}^{bad} f$$
for all $f \in H^0(U_i(\{x\}), \omega^\kappa)$. These morphisms are defined in the same way as the morphism $U_i$, with the same normalization factor. This gives the decomposition of the theorem for $N=1$ at the point $x$. \\
To get the decomposition for $N=2$, one needs to study the operator $U_{i,x}^{bad}$. Let $x_1, \dots, x_{N_1}$ be the elements of the set $U_{i,x}^{bad} (\{x\})$. We can thus decompose the operator $U_{i,x}^{bad}$ in 
$$U_{i,x}^{bad} = \coprod_{j=1}^{N_1} U_{i,x,j}^{bad}$$
where $U_{i,x,j}^{bad}$ corresponds to the point $x_j$. Then for each $1 \leq j \leq N_1$, we have the decomposition of the operator $U_i$ at the point $x_j$ in $U_{i,x_j}^{good} \coprod U_{i,x_j}^{bad}$. One then gets the decomposition of $U_i^2$ at the point $x$
$$U_i^2 = U_i \circ U_{i,x}^{good} \coprod_{j=1}^{N_1} U_{i,x_j}^{good} \circ U_{i,x,j}^{bad}  \coprod_{j=1}^{N_1}  U_{i,x_j}^{bad} \circ U_{i,x,j}^{bad}$$
Repeating this argument, one gets the desired decomposition of the operator $U_i^N$ at the point $x$. Of course, this decomposition does not have a meaning on the whole open $\mathcal{U}$ because the morphisms used to define the different operators will not be finite (the integer $N_1$ depends on the point $x$). But on an adequate subset, this will be the case. That is why one has to construct the subsets $(\mathcal{U}_k (N))_{k \in M_N}$. This construction and the proof of the properties have been done in \cite{Bi1} theorem $4.4.1$.

\subsection{Analytic continuation} \label{analytic}

Let $f$ be a section of the sheaf $\omega^\kappa$ on $\mathcal{U}_0$. We will show that $f$ can be extended to $\mathcal{U}_1$ under a certain condition. More precisely, suppose in this section that $f$ is an eigenform for the Hecke operator $U_{i}$ with eigenvalue $\alpha_{i}$, and that
$$v(\alpha_i) + n_{A_i}  < (1 - 2f\varepsilon) \inf_{\sigma \in S_2} (\kappa_\sigma + \lambda_\sigma)$$
The first step is to extend $f$ to $Deg_i^{-1} (]d_i - t , d_i]) \cap \mathcal{U}_1$. We will use the following proposition.

\begin{prop}
Let $0 < \gamma < 1$ be a rational. Then there exists an integer $N$ such that 
$$U_i^N (Deg_i^{-1} ([d_i - t \gamma,d_i]) \cap \mathcal{U}_1 )  \subset Deg_i^{-1} ([d_i - \varepsilon,d_i]) \cap \mathcal{U}_1$$ 
\end{prop}

\begin{proof}
The key point is that on $Deg_i^{-1} ([d_i - t \gamma,d_i]) \cap \mathcal{U}_1$, the operator $U_i$ increases strictly the degree function $Deg_i$. Since $Deg_i^{-1} ([d_i - t \gamma,d_i]) \cap \mathcal{U}_1$ is a quasi-compact open subset of $X_{Iw,rig}$, one can then apply the argument in \cite{Pi} proposition 2.5.
\end{proof}

\begin{coro}
The overconvergent form $f$ can be extended to $Deg_i^{-1} (]d_i - t , d_i]) \cap \mathcal{U}_1$.
\end{coro}

\begin{proof}
Let $0 < \gamma < 1$ be a rational. By the previous proposition, there exists an integer $N$ such that
$$U_i^N (Deg_i^{-1} ([d_i - t \gamma,d_i]) \cap \mathcal{U}_1 )  \subset Deg_i^{-1} ([d_i - \varepsilon,d_i]) \cap \mathcal{U}_1$$ 
The quantity $\alpha_i^{-N} U_i^N f$ is thus defined on $Deg_i^{-1} ([d_i - t \gamma,d_i]) \cap \mathcal{U}_1$. This formula allows us to extend $f$ to $Deg_i^{-1} (]d_i - t , d_i]) \cap \mathcal{U}_1$.
\end{proof}

The second step is to define some series on $\mathcal{U} := Deg_i^{-1} ([0,d_i-t(1-\alpha)]) \cap \mathcal{U}_1$, for some rational $\alpha$ sufficiently small. We will use the decomposition of the Hecke operator $U_i$. First, we recall the definition of the norm of an operator. If $\mathcal{U}$ is an open subset of $X_{Iw,rig}$, and $T : H^0 (T(\mathcal{U}),\omega^\kappa) \to~H^0(\mathcal{U},\omega^\kappa)$ is an operator, then we define the norm of $T$ as 
$$ \Vert T \Vert := \inf \{r >0 , |Tf|_\mathcal{U} \leq r |f|_{T(\mathcal{U})}, \forall f \in H^0(T(\mathcal{U}),\omega^\kappa) \}$$

\begin{theo}
Suppose that all the operators $U_i^{bad}$ introduced in the theorem $\ref{decompo}$ satisfy the relation 
$$||\alpha_i^{-1} U_i^{bad} || < 1$$
Then it is possible to construct sections $f_N \in H^0 (\mathcal{U}, \omega^\kappa / p^{A_N})$, such that $A_N \to \infty$. Moreover, the functions $f_N$ can be glued together and with the initial form $f$ to give an element of $H^0 ( \mathcal{U}_1,\omega^\kappa)$.
\end{theo}

The construction of the series $f_N$, and the proof of the gluing process have been done in \cite{Bi1} section $4.5$. Let us describe briefly the method. We take an element $\beta >0$, and consider the open subsets $(\mathcal{U}_{k} (N))_{k \in M_N}$ constructed in the theorem \ref{decompo} for this element. Let $K_N$ be the largest element of $M_N$ ; neglecting the bad points for the operator $U_i^N$ gives a function $g_{K_N} \in H^0(\mathcal{U}_{K_N} (N), \omega^\kappa)$ (we refer to \cite{Bi1} definition 4$.5.2$ for the precise definition of the function $g_{K_N}$). Now take a $\beta' < \beta$, and apply again the theorem \ref{decompo} for this other element : we get open subsets $(\mathcal{U}_{k} (N)')_{k \in M_N}$. We consider the subset $\mathcal{U}_{K_N - 1} (N)' \backslash \mathcal{U}_{K_N} (N)'$, and the decomposition of $U_i^N$ on this open. Neglecting the bad points gives a function $g_{K_N - 1} \in H^0(\mathcal{U}_{K_N - 1} (N)' \backslash \mathcal{U}_{K_N} (N)', \omega^\kappa)$. One can then show that $g_{K_N}$ and $g_{K_N - 1}$ can be glued together to give a function modulo $p^{A_N}$, where $A_N$ is an explicit constant. One thus gets a function $f_N \in  H^0(\mathcal{U}_{K_N - 1} (N)', \omega^\kappa / p^{A_N})$ (see \cite{Bi1} proposition $4.5.6$ for more details). Repeating this argument, one thus extends the function $f_N$ to the whole $\mathcal{U}$. The hypothesis on the operators implies that the sequence $(A_N)_N$ tends to infinity. A gluing lemma (\cite{Bi1} proposition $4.5.7$, following \cite{Ka}) then ensures the existence of a function $f \in H^0(\mathcal{U}, \omega^\kappa)$, which can be glued with the initial function $f$, thus getting a section on $\mathcal{U}_1$. \\

We will now prove that, under the assumption made at the beginning of section \ref{analytic}, the condition in the theorem is fulfilled, that is to say that the norm of the operators $\alpha_i^{-1} U_i^{bad}$ is strictly less than $1$. We will split the discussion between the linear and unitary cases.

\subsubsection{Linear case}

We recall that the integer $t$ is equal to $1$ in that case. Let $M$ be a finite extension of $\mathbb{Q}_p$, let $x=(A,\lambda,\iota,\eta,H_\bullet)$ be a point of $\mathcal{U}$ defined over $O_M$, and $L=L_0 \oplus L_0 ^\bot$ be a subgroup of $A[\pi]$, where $L_0$ is an $O_F$-stable subgroup of $A[\pi^+]$ with $A[\pi^+] = H_{A_i} \oplus L_0$ in generic fiber. Let us denote $A' = A/L$, and let $H_{A_i}'$ be the image of $H_{A_i}$ in $A/L$. Thus
$$H_{A_i}' = A[\pi^+]/L_0$$
We suppose that the subgroup $L$ corresponds to a bad point. That is to say that $\deg H_{A_i}' \leq d_i - 1 + \alpha$ for a certain rational $\alpha >0$. We write $\deg_\sigma L_0 = \max(a_\sigma-A_i,0) + l_\sigma$ for all $\sigma \in \Sigma$. As it is shown by the proposition $\ref{minorate}$, $l_\sigma$ is a positive rational for all $\sigma$. We then have for all $\sigma$
$$ \deg_\sigma H_{A_i}' = a_\sigma - \deg_\sigma L_0 = \min(A_i,a_\sigma) - l_\sigma$$
We deduce that $\deg H_{A_i}' = d_i - \sum_{\sigma \in \Sigma} l_\sigma$. The condition of being a bad point gives
$$\sum_{\sigma \in \Sigma} l_\sigma \geq 1 - \alpha$$
Actually, we can control some of the $l_\sigma$. First, we prove the following technical lemma.

\begin{lemm} \label{technic_l}
Let $x=(A,\lambda,\iota,\eta,H_\bullet)$ be a point as before, and let $H \subset A[\pi^+]$ be an $O_F$-stable subgroup. If $H$ is of height $fA_k$ and $\deg H \geq d_k - \varepsilon$, then $\deg_{\sigma} H \geq \min(a_{\sigma},A_k) - \varepsilon$ for all $\sigma \in \Sigma$. \\
If $H$ is of height $f(a+b - A_k)$, and $\deg H \leq \deg A[\pi^+] - d_k + \varepsilon$, then $\deg_\sigma H \leq \max(a_\sigma-A_k,0) + \varepsilon$ for all $\sigma \in \Sigma$.
\end{lemm}

\begin{proof}
Suppose that $H$ is of height $f A_k$ and $\deg H \geq d_k - \varepsilon$. If $\deg_\sigma H < \min(a_\sigma,A_k) - \varepsilon$ for some $\sigma \in \Sigma$, then
$$\deg H = \sum_{\sigma' \in \Sigma} \deg_{\sigma'} H < \min(a_\sigma,A_k) - \varepsilon + \sum_{\sigma' \neq \sigma} \min(a_{\sigma'},A_k) = d_k - \varepsilon$$
and we get a contradiction. \\
If $H$ is of height $f(a+b - A_k)$, and $\deg H \leq \deg A[\pi^+] - d_k + \varepsilon$, then we can apply the previous argument to $A[\pi^+] / H$. We get $\deg_{\sigma} (A[\pi^+] / H) \geq \min(a_{\sigma},A_k) - \varepsilon$ for all $\sigma \in \Sigma$, and therefore $\deg_\sigma H \leq \max(a_\sigma-A_k,0) + \varepsilon$.
\end{proof}

Now we prove a bound for some of the $l_\sigma$.

\begin{lemm}
If $\sigma \in S_1$, we have $l_\sigma \leq \varepsilon$.
\end{lemm}

\begin{proof}
If $\sigma \in \Sigma_0 \cup \Sigma_{s+1}$, then $l_\sigma=0$. If not, $\sigma \in \Sigma_k$, with $1 \leq k \leq s$, $k \neq i$ and $\varepsilon_k = \varepsilon$. Since $x$ is a point of $\mathcal{U}$, we have $\deg H_{A_k} \geq d_k - \varepsilon$. Suppose first that $k < i$. Since $H_{A_k}$ and $L_0$ are disjoint, the morphism $H_{A_k} \to H_{A_k}':=(H_{A_k} + L_0)/ L_0$ is an isomorphism in the generic fiber. Thus by the properties of the degree function $\deg H_{A_k}' \geq \deg H_{A_k} \geq d_k - \varepsilon$. By the lemma $\ref{technic_l}$, we get $\deg_\sigma H_{A_k}' \geq A_k - \varepsilon$. But we also have $\deg_\sigma H_{A_k}' \leq \deg_\sigma H_{A_i}' = \min(A_i,a_\sigma) - l_ \sigma = A_k - l_\sigma$. In conclusion, we get $l_\sigma \leq \varepsilon$. \\
The case $k > i$ can be treated by duality, considering the group $A[\pi^-]$. We can also give a direct argument. Let us denote the group $L_0 / (L_0 \cap H_{A_k})$ by $L_k'$. Since $H_{A_k}$ and $L_0$ generate $A[\pi^+]$, the morphism $ L_k' \to A[\pi^+] / H_{A_k}$ is an isomorphism in generic fiber, so we get $\deg L_k' \leq~\deg (A[\pi^+] / H_{A_k})$. But we have
$$\deg (A[\pi^+] / H_{A_k} ) = \deg A[\pi^+] - \deg H_{A_k} \leq \deg A[\pi^+] - d_k + \epsilon$$
From the lemma $\ref{technic_l}$, we get $\deg_\sigma L_k' \leq \varepsilon$. Since $L_0 \cap H_{A_k}$ is of height $f(A_k-A_i)$, we have
$$l_\sigma = \deg_\sigma L_0 - (A_k - A_i) = \deg_\sigma L_k' + \deg_\sigma (L_0 \cap H_{A_k}) - (A_k-A_i) \leq \varepsilon$$
\end{proof}

\begin{rema}
Actually, we can have more control on the $l_\sigma$. Indeed, suppose that there exists $\sigma \in \Sigma_k \cap S_1$. If $k<i$, then we have $l_\sigma \leq \varepsilon$ for all $\sigma \in \Sigma_j$, with $j \leq k$. If $k>i$, then we have $l_\sigma \leq \varepsilon$ for all $\sigma \in \Sigma_j$, with $j \geq k$.
\end{rema}

Putting together all the calculations made, we get the following result.

\begin{prop}
We have
$$\sum_{\sigma \in S_2} l_\sigma \geq 1 - \alpha - f \varepsilon$$
\end{prop}

\begin{proof}
If $\sigma \in S_1$, we have $l_\sigma \leq \varepsilon$, and we also have $\sum_{\sigma \in \Sigma} l_\sigma \geq 1 - \alpha$. Thus
$$\sum_{\sigma \in S_2} l_\sigma \geq 1 - \alpha - \sum_{\sigma \in S_1} l_\sigma \geq 1 - \alpha - f \varepsilon$$
\end{proof}

We can now prove the bound for the norm of the operator $U_i^{bad}$.

\begin{prop}
We have 
$$|| \alpha_i^{-1} U_i^{bad} || \leq p^{ v(\alpha_i) + n_{A_i} - (1 - \alpha - 2f \varepsilon) \inf_{\sigma \in S_2} (\kappa_\sigma + \lambda_\sigma)    }$$
\end{prop} 

\begin{proof}
The term $p^{v(\alpha_i)}$ is the norm of the element $\alpha_i^{-1}$. Recall also the normalization factor for the Hecke operator
$$N_{A_i} = n_{A_i} + \sum_{\sigma \in \Sigma_\pi} \Big( \max(a_\sigma-A_i,0) \kappa_\sigma + \max(A_i-a_\sigma,0) \lambda_\sigma \Big) $$
The first term will come in the bound, and the second term will be canceled out by the bounds for the partial degrees of $L_0$ and $L_0^\bot$. \\
Indeed, let us calculate the norm of the morphism $\omega_{A/L}^\kappa \to \omega_{A}^\kappa$. Let $\kappa_1$ be the weight defined by $((\kappa_\sigma, \dots, \kappa_\sigma) , (\lambda_\sigma, \dots, \lambda_\sigma))_{\sigma \in \Sigma}$, and $\kappa_2 = \kappa - \kappa_1$. Thus $\omega^\kappa = \omega^{\kappa_1} \otimes \omega^{\kappa_2}$ and $\omega^{\kappa_1}$ is a line bundle. Since $\kappa_2$ has non-negative coefficients, the morphism $\omega_{A/L}^{\kappa_2} \to \omega_{A}^{\kappa_2}$ has a norm less than $1$. It then suffices to study the morphism $\omega_{A/L}^{\kappa_1} \to \omega_{A}^{\kappa_1}$. But we have 
$$\omega_A^{\kappa_1} = \bigotimes_{\sigma \in \Sigma} (\det \omega_{A,\sigma}^+)^{\kappa_\sigma} \otimes (\det \omega_{A,\sigma}^-)^{\lambda_\sigma}$$
where $\omega_{A,\sigma}^+$ is the sub-module of $\omega_A$ where $O_F$ acts on $\sigma^+$, and similarly for $\omega_{A,\sigma}^-$. We recall that $L = L_0 \oplus L_0^\bot$. The norm of the map $\omega_{A/L}^{\kappa_1} \to \omega_{A}^{\kappa_1}$ is exactly $p^A$, with 
$$A = - \sum_{\sigma \in \Sigma} \left( \kappa_\sigma \deg_\sigma L_0 + \lambda_\sigma \deg_\sigma L_0^\bot \right)$$
But recall that $\deg_\sigma L_0 = \max(a_\sigma - A_i,0) + l_\sigma$, and that 
$$\deg_\sigma L_0^\bot = A_i - a_\sigma + \deg_\sigma L_0 = \max(A_i-~a_\sigma,0) + l_\sigma$$ 
Thus
$$A = - \sum_{\sigma \in \Sigma} \left( \max(a_\sigma - A_i,0) \kappa_\sigma + \max(A_i-~a_\sigma,0) \lambda_\sigma + l_\sigma (\kappa_\sigma + \lambda_\sigma) \right) = - (N_{A_i} - n_{A_i}) + B$$
with
$$B= - \sum_{\sigma \in \Sigma} \left( l_\sigma (\kappa_\sigma + \lambda_\sigma) \right) \leq -\sum_{\sigma \in S_2} \left( l_\sigma (\kappa_\sigma + \lambda_\sigma) \right) \leq -\inf_{\sigma \in S_2} (\kappa_\sigma + \lambda_\sigma) \sum_{\sigma \in S_2} l_\sigma \leq -(1 - \alpha - f \varepsilon)  \inf_{\sigma \in S_2} (\kappa_\sigma + \lambda_\sigma)$$
Hence the bound for the operator $\alpha_i^{-1} U_i^{bad}$, noting that $1 - \alpha - f \varepsilon \geq 1 - \alpha - 2f \varepsilon$.
\end{proof}

Since we have made the assumption
$$v(\alpha_i) + n_{A_i}  < (1 - 2f\varepsilon) \inf_{\sigma \in S_2} (\kappa_\sigma + \lambda_\sigma)$$
we see that if $\alpha$ is small enough, the norm of the operator $\alpha_i^{-1} U_i^{bad}$ is strictly less than $1$.

\begin{rema}
Actually, we only needed the weaker hypothesis
$$v(\alpha_i) + n_{A_i}  < (1 - f\varepsilon) \inf_{\sigma \in S_2} (\kappa_\sigma + \lambda_\sigma)$$
but we used it to have the same condition on both linear and unitary cases.
\end{rema}

\subsubsection{Unitary case} 

We now turn to the unitary case. We first assume that $i < (a+b)/2$, so that the integer $t$ is equal to $1$. Let $M$ be a finite extension of $\mathbb{Q}_p$, let $x=(A,\lambda,\iota,\eta,H_\bullet)$ be a point of $\mathcal{U}$ defined over $O_M$, and $L$ be a totally isotropic $O_F$-stable subgroup of $A[\pi^2]$, such that we have $A[\pi] = H_i \oplus L[\pi]$. Let us denote $A' = A/L$, and let $H_{A_i}'$ be the image of $H_{A_i}$ in $A/L$. Thus
$$H_{A_i}' = (A[\pi]+L)/L$$
We suppose that the subgroup $L$ corresponds to a bad point. That is to say that $\deg H_{A_i}' \leq d_i - 1 + \alpha$ for a certain rational $\alpha >0$. We write $\deg_{\sigma_1} L = \max(a_\sigma-A_i,0) + l_{\sigma}$, and 
$$\deg_{\sigma_2} L = b_\sigma - a_\sigma + \deg_{\sigma_1} L = \max(b_\sigma - A_i, b_\sigma - a_\sigma ) + l_{\sigma}$$
for all $\sigma \in \Sigma$. As it is shown by the proposition $\ref{minorate_u}$, $l_{\sigma}$ is a positive rational for all $\sigma$. We then have $ \deg H_{A_i}' = 2fA_i - \deg (L/L[\pi])$, therefore 
$$\deg (L/L[\pi]) \geq 2fA_i - d_i +1-\alpha$$
We also have by the lemma $\ref{isotrop}$ the inequality $\deg (L/L[\pi]) \leq 2fA_i - \deg A[\pi] + \deg L[\pi]$. Thus
$$\deg L[\pi] \geq \deg A[\pi] - d_i + 1 - \alpha$$
We conclude that
$$\deg L \geq 2f A_i - 2 d_i + \deg A[\pi] + 2(1 - \alpha)$$
By an explicit computation, one checks that the equality 
$$2f A_i - 2 d_i + \deg A[\pi] = \sum_{\sigma \in \Sigma} (\max(a_\sigma-A_i,0) + \max(b_\sigma - A_i, b_\sigma - a_\sigma ))$$
We conclude that the condition of being a bad point gives
$$\sum_{\sigma \in \Sigma} l_\sigma \geq 1 - \alpha$$

Actually, we can control some of the $l_\sigma$. First, we prove the following technical lemma.

\begin{lemm} \label{technic_u}
Let $x=(A,\lambda,\iota,\eta,H_\bullet)$ be a point as before, and let $H \subset A[\pi]$ be an $O_F$-stable subgroup. If $H$ is of height $2fA_k$ and $\deg H \geq d_k - \varepsilon$, then $\deg_{\sigma_1} H \geq \min(a_{\sigma},A_k) - \varepsilon$ for all $\sigma \in \Sigma$. \\
If $H$ is of height $2f(a+b - A_k)$, and $\deg H \leq \deg A[\pi] - d_k + \varepsilon$, then $\deg_{\sigma_1} H \leq \max(a_\sigma-A_k,0) + \varepsilon$ for all $\sigma \in \Sigma$.
\end{lemm}

\begin{proof}
Suppose that $H$ is of height $2f A_k$ and $\deg H \geq d_k - \varepsilon$. If $\deg_{\sigma_1} H < \min(a_\sigma,A_k) - \varepsilon$ for some $\sigma \in \Sigma$, then
$$\deg H = \sum_{\sigma' \in \Sigma} (\deg_{\sigma_1'} H + \deg_{\sigma_2'} H) < \min(a_\sigma,A_k) - \varepsilon + A_k + \sum_{\sigma' \neq \sigma} (\min(a_{\sigma'},A_k) + A_k) = d_k - \varepsilon$$
and we get a contradiction. \\
If $H$ is of height $2f(a+b - A_k)$, and $\deg H \leq \deg A[\pi] - d_k + \varepsilon$, then we can apply the previous argument to $A[\pi] / H$. We get $\deg_{\sigma_1} (A[\pi] / H) \geq \min(a_{\sigma},A_k) - \varepsilon$ for all $\sigma \in \Sigma$, and therefore $\deg_{\sigma_1} H \leq \max(a_\sigma-A_k,0) + \varepsilon$.
\end{proof}

\noindent Now we prove a bound for some of the $l_\sigma$.

\begin{lemm}
If $\sigma \in S_1$, we have $l_\sigma \leq 2\varepsilon$.
\end{lemm}

\begin{proof}
If $\sigma \in \Sigma_0$, then $l_\sigma=0$. If not, $\sigma \in \Sigma_k$, with $1 \leq k \leq s$, $k \neq i$ and $\varepsilon_k = \varepsilon$. Since $x$ is a point of $\mathcal{U}$, we have $\deg H_{A_k} \geq d_k - \varepsilon$. Suppose first that $k < i$. Since $H_{A_k}$ and $L[\pi]$ are disjoint, the morphism $H_{A_k} \to H_{A_k}':=(H_{A_k} + L[\pi])/ L[\pi]$ is an isomorphism in the generic fiber. Thus by the properties of the degree function $\deg H_{A_k}' \geq \deg H_{A_k} \geq d_k - \varepsilon$. By the lemma $\ref{technic_u}$, we have $\deg_{\sigma_1} H_{A_k}' \geq A_k - \varepsilon$. But we also have $\deg_{\sigma_1} H_{A_k}' \leq \deg_{\sigma_1} A[\pi] / L[\pi] = a_\sigma - \deg_{\sigma_1} L[\pi]$, and therefore $\deg_{\sigma_1} L[\pi] \leq \varepsilon$. Next, we study $H_{A_k}'':= (H_{A_k}' + L/L[\pi]) / (L/L[\pi])$. Using the lemma $\ref{technic_u}$, we get
$$A_k - \varepsilon \leq \deg_{\sigma_1} H_{A_k''} \leq A_k - \deg_{\sigma_1} L/L[\pi]$$
Therefore $\deg_{\sigma_1} (L/L[\pi]) \leq \varepsilon$. In conclusion, we get $l_\sigma \leq 2 \varepsilon$. \\
The case $k > i$ cannot be treated by duality in this case. Let us denote the group $L[\pi] / (L[\pi] \cap H_{A_k})$ by $L_k'$. Since $H_{A_k}$ and $L[\pi]$ generate $A[\pi]$, the morphism $ L_k' \to A[\pi] / H_{A_k}$ is an isomorphism in generic fiber, so we get $\deg L_k' \leq \deg (A[\pi] / H_{A_k})$. But we have
$$\deg (A[\pi] / H_{A_k} ) = \deg A[\pi] - \deg H_{A_k} \leq \deg A[\pi] - d_k + \epsilon$$
The lemma $\ref{technic_u}$ gives $\deg_{\sigma_1} L_k' \leq \varepsilon$. Since $L[\pi] \cap H_{A_k}$ is of height $2f(A_k-A_i)$, we have
$$\deg_{\sigma_1} L[\pi] = \deg_{\sigma_1} L_k' + \deg_{\sigma_1} L[\pi] \cap H_{A_k} \leq \varepsilon + A_k - A_i$$
Now we study $A' := A/L[\pi]$ and set $H_{A_k}' := (\pi^{-1} (H_{A_k} \cap L[\pi]))/L[\pi]$. If $H_{A_k}''$ is the image of the morphism $H_{A_k}' \to A'[\pi] / (L/L[\pi])$, then we get $\deg H_{A_k}'' \geq d_k - \varepsilon$. The lemma $\ref{technic_u}$ gives
$$A_k - \varepsilon \leq \deg_{\sigma_1} H_{A_k}'' \leq A_k - \deg_{\sigma_1} (L/L[\pi])$$
We deduce that $\deg_{\sigma_1} L/L[\pi] \leq \varepsilon$. Finally, we have
$$l_\sigma = \deg_{\sigma_1} L - (A_k-A_i) \leq 2 \varepsilon$$
\end{proof}

\begin{rema}
Actually, we can have more control on the $l_\sigma$. Indeed, suppose that there exists $\sigma \in \Sigma_k \cap S_1$.
If $k<i$, then we have $l_\sigma \leq \varepsilon$ for all $\sigma \in \Sigma_j$ with $j \leq k$. If $k>i$, then we have $l_\sigma \leq \varepsilon$ for all $\sigma \in \Sigma_j$ with $j \geq k$.
\end{rema}

Putting together all the calculations made, we get the following result.

\begin{prop}
We have
$$\sum_{\sigma \in S_2} l_\sigma \geq 1 - \alpha - 2f \varepsilon$$
\end{prop}

\begin{proof}
If $\sigma \in S_1$, we have $l_\sigma \leq 2\varepsilon$, and we also have $\sum_{\sigma \in \Sigma} l_\sigma \geq 1 - \alpha$. Thus
$$\sum_{\sigma \in S_2} l_\sigma \geq 1 - \alpha - \sum_{\sigma \in S_1} l_\sigma \geq 1 - \alpha - 2f \varepsilon$$
\end{proof}

We can now prove the bound for the norm of the operator $U_i^{bad}$.

\begin{prop} \label{bound}
We have 
$$|| \alpha_i^{-1} U_i^{bad} || \leq p^{ v(\alpha_i) + n_{A_i} - (1 - \alpha - 2f \varepsilon) \inf_{\sigma \in S_2} (\kappa_\sigma + \lambda_\sigma)    }$$
\end{prop} 

\begin{proof}
The proof is exactly the same as in the linear case.
\end{proof}

Since we have made the assumption
$$v(\alpha_i) + n_{A_i}  < (1 - 2f\varepsilon) \inf_{\sigma \in S_2} (\kappa_\sigma + \lambda_\sigma)$$
we see that if $\alpha$ is small enough, the norm of the operator $\alpha_i^{-1} U_i^{bad}$ is strictly less than $1$. \\

We now deal with the case $i=(a+b)/2$ : the integer $t$ is now equal to $2$. The subgroup $L$ is now a maximal totally isotropic subgroup of $A[\pi]$. We write $\deg_{\sigma_1} L = l_\sigma$, and we have $\deg_{\sigma_2} L = (b_\sigma - a_\sigma)/2 + l_\sigma$, for all $\sigma \in \Sigma$. Recall that we have the relation
$$\deg H_{(a+b)/2} ' = f(a+b) - \deg L = \sum_{\sigma \in \Sigma} (a_\sigma + \frac{(a+b)}{2} - 2l_\sigma  )$$
so that 
$$d_{(a+b)/2} - \deg H_{(a+b)/2} ' = 2 \sum_{\sigma \in \Sigma} l_\sigma$$
The condition of being a bad point gives
$$\sum_{\sigma \in \Sigma} l_\sigma \geq 1 - \alpha$$
for a certain $\alpha >0$. The calculations are now exactly the same as previously, and the bound obtained in the proposition $\ref{bound}$ is still valid in the case $i=(a+b)/2$.

\subsubsection{Conclusion}

We now recall the analytic continuation result we got, putting together all the results of the past sections. \\
Recall that we have defined subsets
$$\mathcal{U}_0 := Deg^{-1} ([d_1-\varepsilon_1,d_1] \times \dots \times [d_s-\varepsilon_s,d_s]) = \bigcap_{k=1}^s Deg_k^{-1}  [d_k-\varepsilon_k,d_k]$$
$$\mathcal{U}_1 := \bigcap_{k \neq i} Deg_k^{-1}  [d_k-\varepsilon_k,d_k]$$ 
We also have defined a partition $\Sigma = S_1 \coprod S_2$. The result of this section is thus contained in the following theorem.

\begin{theo}
Let $\kappa$ be a weight, and $f$ a section of $\omega^\kappa$ on $\mathcal{U}_0$. Suppose that $f$ is an eigenform for the Hecke operators $U_{i}$, with eigenvalue $\alpha_i$ for all $1 \leq i \leq s$, and that we have the relations
$$n_{A_i} + v(\alpha_i) < (1-2f\varepsilon) \inf_{\sigma \in S_2} (\kappa_\sigma + \lambda_\sigma)$$
Then $f$ can be extended to a section of $\omega^\kappa$ on $\mathcal{U}_1$.
\end{theo}

\subsection{The classicality theorem}

We recall that we have made the assumption that $p$ is inert in $F_0$. We have defined integers $(A_i)_{1 \leq i \leq s}$, and to each of these integers corresponds a canonical subgroup, and a relevant Hecke operator. We also have a partition
$$\Sigma = \Sigma_0 \coprod_{j=1}^s \Sigma_j \coprod \Sigma_{s+1}$$
We can now state the classicality result.

\begin{theo} \label{final_theo}
Let $f$ be an overconvergent modular form of weight $\kappa$. Suppose that $f$ is an eigenform for the Hecke operators $U_{\pi,A_i}$, with eigenvalue $\alpha_i$, and that we have the relations
$$n_{A_i} + v(\alpha_i) < \inf_{\sigma \in \Sigma_i} (\kappa_\sigma + \lambda_\sigma)$$
for $1 \leq i \leq s$. Then $f$ is classical.
\end{theo}

Before giving the proof of the theorem, let us make the conditions explicit. In the case $L$, the conditions become
$$v(\alpha_i) + \sum_{j=1}^f \min(a_j,A_i) \min (b_j,B_i) < \inf_{\sigma \in \Sigma_i} (\kappa_\sigma + \lambda_\sigma)$$
In the special case where all the $(a_i)$ are distinct and different from $0$ and $a+b$, we have $s=d$ conditions, which can be written
$$v(\alpha_\sigma) + \sum_{\sigma' \in \Sigma} \min(a_\sigma,a_{\sigma'}) \min (b_\sigma,b_{\sigma'}) < \kappa_{\sigma} + \lambda_{\sigma}$$
where $\alpha_{\sigma}$ is the eigenvalue of $U_{\pi,a_\sigma}$. \\
In the case $U$, they become
$$v(\alpha_i) + \sum_{j=1}^f (a+b)\min(a_j,A_i) < \inf_{\sigma \in \Sigma_i} (\kappa_\sigma + \lambda_\sigma)$$
if $A_i < (a+b)/2$, and
$$v(\alpha_i) + \sum_{j=1}^f \frac{(a+b)}{2} a_j < \inf_{\sigma \in \Sigma_i} (\kappa_\sigma + \lambda_\sigma)$$
if $A_i = (a+b)/2$. \\
In the special case where all the $(a_i)$ are distinct and different from $0$ and $(a+b)/2$, we have $s=d$ conditions, which can be written
$$v(\alpha_\sigma) + \sum_{\sigma' \in \Sigma} (a+b)\min(a_\sigma,a_{\sigma'}) < \kappa_{\sigma} + \lambda_{\sigma}$$
where $\alpha_{\sigma}$ is the eigenvalue of $U_{\pi,a_\sigma}$. \\
$ $\\
Of course, in the case where the ordinary locus in non empty, we find the same conditions as in \cite{Bi1}. In the case $L$, the condition of ordinariness is $a_{\sigma}=a$, $b_\sigma=b$ for some couple $(a,b)$ and for all $\sigma \in \Sigma$. There is one relevant Hecke operator, $U_{\pi,a}$, and the classicality condition is
$$fab + v(\alpha) < \inf_{\sigma \in \Sigma}  (\kappa_\sigma + \lambda_\sigma)$$
where $\alpha$ is the eigenvalue of $U_{\pi,a}$. \\
In the case $U$, the condition of ordinariness is $$a_{\sigma} = b_{\sigma} = (a+b)/2$$ for all $\sigma \in \Sigma$. There is one relevant Hecke operator, $U_{\pi,(a+b)/2}$, and the condition is 
$$f\frac{(a+b)^2}{4} + v(\alpha) < \inf_{\sigma \in \Sigma}  (\kappa_\sigma + \lambda_\sigma)$$
where $\alpha$ is the eigenvalue of $U_{\pi,(a+b)/2}$. 

\begin{rema}
Since we need to use all the Hecke operators $U_{\pi,A_i}$ for the classicality result, maybe the relevant operator is a product of these ones. For example, in the linear case the operator $\prod_\sigma U_{\pi,a_\sigma}$ parametrizes complements of (a lifting of) the kernel of the $f^{th}$-power of the Frobenius on the $\mu$-ordinary locus. 
\end{rema}

\subsubsection{Proof of the theorem}

Let us now prove the theorem. The overconvergent form $f$ is a section of the sheaf $\omega^\kappa$ defined over
$$\mathcal{V}_0 := \bigcap_{i=1}^s Deg_i^{-1} ([d_i-\varepsilon, d_i])$$
for some $\varepsilon >0$. Of course, $\varepsilon$ can be taken as small as we want.
Let us note $K_i = \inf_{\sigma \in \Sigma_i} (\kappa_\sigma + \lambda_\sigma)$, for all $1 \leq i \leq s$. We order the elements $(K_i)_{1 \leq i \leq s}$ by decreasing order : we have
$$K_{i_1} \geq K_{i_2} \geq \dots \geq K_{i_s}$$
We will use the analytic continuation theorem successively for the operators $U_{i_1}, \dots, U_{i_s}$, in that order. \\
We first consider the operator $U_{i_1}$ (we recall that we noted $U_i := U_{\pi,A_i}$). We take all the rationals $\varepsilon_k$ to be equal to $\varepsilon$. In that case, $S_2 = \Sigma_{i_1}$. We can apply the analytic continuation theorem if the condition 
$$n_{A_{i_1}} + v(\alpha_{i_1}) < (1-2f\varepsilon) \inf_{\sigma \in S_2} (\kappa_\sigma + \lambda_{\sigma})$$
is fulfilled. But $\inf_{\sigma \in S_2} (\kappa_\sigma + \lambda_{\sigma}) = K_{i_1}$, and we have by hypothesis
$$n_{A_{i_1}} + v(\alpha_{i_1}) < K_{i_1}$$
If $\varepsilon$ is small enough, then we can apply the theorem. We can thus extend $f$ to 
$$\mathcal{V}_1 := \bigcap_{i \neq i_1} Deg_i^{-1} ([d_i-\varepsilon, d_i])$$
We then use the operator $U_{i_2}$. In this case, we take all the rationals $\varepsilon_k$ to be equal to $\varepsilon$, except $\varepsilon_{i_1} = d_{i_1}$. In that case, $S_2 = \Sigma_{i_1} \cup \Sigma_{i_2}$. We can apply the analytic continuation theorem if the condition 
$$n_{A_{i_2}} + v(\alpha_{i_2}) < (1-2f\varepsilon) \inf_{\sigma \in S_2} (\kappa_\sigma + \lambda_{\sigma})$$
is fulfilled. But $\inf_{\sigma \in S_2} (\kappa_\sigma + \lambda_{\sigma}) = \inf (K_{i_1},K_{i_2}) = K_{i_2}$, and we have by hypothesis
$$n_{A_{i_2}} + v(\alpha_{i_2}) < K_{i_2}$$
If $\varepsilon$ is small enough, then we can apply the theorem. We extend $f$ to 
$$\mathcal{V}_2 := \bigcap_{i \notin \{ i_1,i_2 \}} Deg_i^{-1} ([d_i-\varepsilon, d_i])$$
Repeating this argument, we can extend the overconvergent form $f$ to the whole rigid variety $X_{Iw,rig}$.
We conclude by applying a Koecher's principle, and a GAGA theorem, which proves that the space
$$H^0(X_{Iw,rig},\omega^\kappa)$$
consists of classical modular forms. This will be done in the next section.

\subsubsection{Compactifications and Koecher's principle} \label{compact}

To complete the proof of the theorem, we need to prove a Koecher's principle, and to introduce compactifications of the Shimura variety.

\begin{prop}
There exists a toro\"idal compactification $\overline{X}_{Iw}$ of $X_{Iw}$ defined over $O_K$. It is a proper scheme, and the sheaf $\omega^\kappa$ extends to $\overline{X}_{Iw}$. 
\end{prop}

The construction of the compactification of $X$ has been done in \cite{La}, and the one for $X_{Iw}$ follows from \cite{La2} section $3$ (see also \cite{Bi1} section $5.1$).
One can also construct the minimal compactification $X_{Iw}^*$ of $X_{Iw}$. The Koecher's principle states that, under a certain condition, the sections on $X_{Iw}$ automatically extend to the toro\"idal compactification. The following proposition follows from \cite{La2} Theorem $8.7$ (see also \cite{Bi1} section $5.2$).

\begin{prop} 
Suppose that the codimension of the boundary of $X_{Iw}^*$ is greater than $2$. Then for any $O_K$-algebra $R$ the restriction map
$$H^0( \overline{X}_{Iw} \times R, \omega^\kappa) \to H^0( X_{Iw} \times R, \omega^\kappa)$$
is an isomorphism.
\end{prop}

From this, we deduce a rigid Koecher's principle (under the same dimension assumption) : 
$$H^0( \overline{X}_{Iw,rig}, \omega^\kappa) \simeq H^0( X_{Iw,rig}, \omega^\kappa)$$
where $\overline{X}_{Iw,rig}$ is the rigid space associated to $\overline{X}_{Iw}$. Finally, since $\overline{X}_{Iw}$ is proper, we have a GAGA theorem.

\begin{prop}
The analytification morphism
$$H^0( \overline{X}_{Iw} \times K, \omega^\kappa) \to H^0( \overline{X}_{Iw,rig}, \omega^\kappa)$$
is an isomorphism.
\end{prop}

To conclude, we need to make explicit the dimension condition. We have the following cases.

\begin{itemize}
\item If there exists $\sigma$ such that $a_\sigma b_\sigma =0$, then the varieties $X$ and $X_{Iw}$ are compact (\cite{La}, remark $5.3.3.2$).
\item If it is not the case, the codimension of the boundary of $X_{Iw}^*$ is equal to $d(a+b-1)$. Indeed, the dimension of the variety is equal to $\sum_{\sigma} a_\sigma b_\sigma$. From \cite{La} theorem $7.2.4.1$, there is a stratification of $X_{Iw}^*$, and any top dimensional strata of the boundary is isomorphic to a Shimura variety with signatures $(a_\sigma - 1, b_\sigma - 1)$, hence has dimension $\sum_\sigma (a_\sigma - 1)(b_\sigma - 1)$. \\
If $d>1$ or $a+b \geq 3$, then the condition for the Koecher's principle is fulfilled.
\item If $d=a=b=1$, we cannot apply the Koecher's principle.
\end{itemize}

Since we have excluded the third case by the hypothesis $\ref{notcurve}$, it follows from what have been said before that we have an isomorphism
$$H^0( \overline{X}_{Iw} \times K, \omega^\kappa) \simeq H^0( X_{Iw,rig}, \omega^\kappa)$$
that is to say that the space $H^0( X_{Iw,rig}, \omega^\kappa)$ consists of classical modular forms. If the vartiety is compact, this is only the GAGA theorem, and in the non-compact case, it is a combination of the rigid Koecher's principle and the GAGA theorem. This concludes the proof of the theorem. \\
In the exceptional remaining case, the variety is essentially a modular curve. To prove the classicality theorem, one has to take the cusps into account in the series constructed. We refer to \cite{Ka} for more details.

\section{Case with several primes above $p$} \label{several}

In this section, we no longer assume that there is only one prime of $F_0$ above $p$. We will define the degree functions, the overconvergent modular forms, and prove the classicality result. \\

Let $\mathcal{P}$ be the set of primes of $F_0$ above $p$. We write $\mathcal{P}=\{\pi_1, \dots, \pi_g \}$. Recall that we have defined the Shimura variety of Iwahori level at $p$ 
$$X_{Iw} = X_{Iw,\pi_1} \times_X X_{Iw,\pi_2} \times_X \dots \times_X X_{Iw,\pi_g}$$

We will denote $X_{Iw,rig}$ the rigid space associated to the scheme $X_{Iw}$. Let $\pi$ be an element of $\mathcal{P}$. In the previous sections, we have defined an integer $s_\pi$, integers $A_{\pi,1} < \dots < A_{\pi,s_\pi}$, other integers $d_{\pi,1}, \dots, d_{\pi,s_\pi}$, together with a degree function

$$Deg_\pi : X_{Iw,\pi,rig} \to \prod_{i=1}^{s_\pi} [0,d_{\pi,i}]$$
We will define the degree function on $X_{Iw,rig}$ by 
\begin{displaymath}
\begin{aligned}
Deg : X_{Iw,rig} & \to & \prod_{\pi \in \mathcal{P}} \prod_{i=1}^{s_\pi} [0,d_{\pi,i}] \\
x & \to & (Deg_{\pi} (p_\pi(x)))_{\pi \in \mathcal{P}}
\end{aligned}
\end{displaymath}
Here $p_\pi$ is the projection $X_{Iw,rig} \to X_{Iw,\pi,rig}$. \\
Let $X_{Iw}^{\mu-ord-mult}$ be the space $Deg^{-1} (\{ d_{\pi,i} \})$. Let us fix a weight $\kappa$.

\begin{defi}
The space of overconvergent modular forms of weight $\kappa$ is defined as 
$$M^\dagger := colim_\mathcal{V} H^0(\mathcal{V}, \omega^\kappa)$$
where $\mathcal{V}$ runs over the strict neighborhoods of $X_{Iw}^{\mu-ord-mult}$ in $X_{Iw,rig}$.
\end{defi}

We have defined Hecke operators $U_{\pi,A_{\pi,1}}, \dots, U_{\pi,A_{\pi,s_\pi}}$. There is a slight ambiguity for the polarization $\lambda'$. If $\pi$ is the only place above $p$, one can take $p \cdot \lambda$, which is a prime to $p$ polarization. In general, $\lambda'$ is not well defined. Let $x$ be a totally positive element of $O_F$ such that $v_{\pi'} (x)$ equals $1$ if $\pi'=\pi$, and $0$ if $\pi'$ is a place of $F_0$ above $p$ different from $\pi$. Then one takes for $\lambda'$ the polarization $x \cdot \lambda$. \\
These operators act both on classical and overconvergent modular forms. We can finally state and prove the classicality theorem.

\begin{theo}
Let $f$ be an overconvergent modular form. Suppose that $f$ is an eigenform for the Hecke operators $U_{\pi,A_{\pi,i}}$, with eigenvalue $\alpha_{\pi,i}$. If we have the relations
$$n_{\pi,A_{\pi,i}} + v(\alpha_{\pi,i}) < \inf_{\sigma \in \Sigma_{\pi,i}} (\kappa_\sigma + \lambda_\sigma)$$
for all $\pi \in \mathcal{P}$, $1 \leq i \leq s_\pi$, then $f$ is classical.
\end{theo}

We recall that $n_{\pi,i}$ is the constant term of the normalization factor for the Hecke operator $U_{\pi,i}$, and that $\Sigma_{\pi,i} = \{ \sigma \in \Sigma_\pi, a_\sigma = A_{\pi,i} \}$.

\begin{proof}
We will use for each place $\pi$ the analytic continuation theorem we got in the previous section to extend the modular form to the whole rigid variety, and deduce its classicality. Luckily, the order of the places $\pi$ is not important here. \\
We start with an overconvergent form $f$. It is a section of $\omega^\kappa$ defined on a set of the form
$$Deg^{-1} ( \prod_{\pi \in \mathcal{P}} \prod_{i=1}^{s_\pi} [d_{\pi,i} - \varepsilon, d_{\pi,i}] )$$
First, we consider the place $\pi_1$. Using the Hecke operators, and the relations satisfied by the eigenvalues, by the previous section, one can extend $f$ to 
$$Deg^{-1} ( \prod_{i=1}^{s_{\pi_1}} [0, d_{\pi_1,i}]   \times \prod_{\pi \neq \pi_1} \prod_{i=1}^{s_\pi} [d_{\pi,i} - \varepsilon, d_{\pi,i}] )$$
Using successively this argument with the different places, one can extend $f$ to the whole  rigid variety $X_{Iw,rig}$. Using a Koecher's principle (if the variety is not compact), and a GAGA theorem as in the section $\ref{compact}$, one can prove that $f$ is a classical modular form.
\end{proof}

\bibliographystyle{amsalpha}

\end{document}